\documentclass[11pt]{amsart}
\usepackage{amsmath,,amsthm, amssymb, amscd,amsxtra, esint}

\usepackage[dvips]{graphics,epsfig} 

\allowdisplaybreaks[2]

\newtheorem{theorem}{Theorem} [section] 
\newtheorem{maintheorem}{Theorem} 
\newtheorem{lemma}[theorem]{Lemma} 
\newtheorem{proposition}[theorem]{Proposition} 
\newtheorem{remark}[theorem]{Remark}

\numberwithin{equation}{section} \numberwithin{theorem}{section}

\DeclareMathOperator*{\intt}{\int}      

\newcommand{\avint}{\fint}

\newcommand{\I}{\hspace{0.5mm}\text{I}\hspace{0.5mm}}

\newcommand{\II}{\text{I \hspace{-2.8mm} I} }

\newcommand{\noi}{
\noindent} 
\newcommand{\Z}{\mathbb{Z}} 
\newcommand{\R}{\mathbb{R}} 
 
\newcommand{\T}{\mathbb{T}} 
\newcommand{\N}{\mathcal{N}}

\newcommand{\al}{\alpha} 
\newcommand{\dl}{\delta} 
 
\newcommand{\eps}{\varepsilon} 
 
\newcommand{\g}{\gamma} 
\newcommand{\G}{\Gamma} 
\newcommand{\ld}{\lambda} 
 
\newcommand{\s}{\sigma} 
\newcommand{\ft}{\widehat} 
\newcommand{\wt}{\widetilde} 
\newcommand{\cj}{\overline} 
\newcommand{\dx}{
\partial_x} 
\newcommand{\dt}{
\partial_t}

\newcommand{\jb}[1] {\langle #1 \rangle}

\begin{document}

\title [Almost Sure Well-Posedness of NLS below $L^2$] {\bf Almost sure well-posedness of the cubic nonlinear Schr\"odinger equation below $L^2 (\mathbb{T})$}

\author{James Colliander, Tadahiro Oh}

\address{James Colliander\\
Department of Mathematics\\
University of Toronto\\
40 St. George St, Toronto, ON M5S 2E4, Canada}
\thanks{J.C. was supported in part by NSERC grant RGP250233-07.}
\email{colliand@math.toronto.edu}

\address{Tadahiro Oh\\
Department of Mathematics\\
University of Toronto\\
40 St. George St, Toronto, ON M5S 2E4, Canada}

\curraddr{Department of Mathematics\\
Princeton University\\
Fine Hall, Washington Rd\\
Princeton, NJ 08544-1000, USA}

\email{hirooh@math.princeton.edu}

\subjclass[2010]{35Q55, 37K05, 37L50, 37L40}

\keywords{Schr\"odinger equation; NLS, well-posedness; invariant measures; ill-posedness}
\begin{abstract}
We consider the Cauchy problem for the one-dimensional periodic cubic nonlinear Schr\"odinger 
equation (NLS) with initial data below $L^2$. 
In particular, we exhibit nonlinear smoothing when the initial data are randomized. 
Then, we prove local well-posedness of NLS almost surely for the initial data 
in the support of the canonical Gaussian measures on $H^s(\mathbb{T})$ for each $s > -\frac{1}{3}$,
and global well-posedness for each $s > -\frac{1}{12}$.

\end{abstract}


\maketitle

\tableofcontents

\newpage

\section{Introduction}

We consider the Cauchy problem for the one-dimensional periodic cubic nonlinear Schr\"odinger equation (NLS): 
\begin{equation}
	\label{NLS1} 
	\begin{cases}
		i u_t - u_{xx} \pm u |u|^2 =0 \\
		u|_{t= 0} = u_0, ~x \in \T = \R / 2 \pi \Z.
	\end{cases}
\end{equation}

\noi
We first establish almost sure local well-posedness 
for{\footnote{We actually consider the Wick ordered version \eqref{NLS2} instead of \eqref{NLS1} below.}} \eqref{NLS1} 
with respect to the canonical Gaussian measure supported on $H^s (\T)$ in the range $-\frac{1}{3} < s < 0$. 
Then, we establish almost sure global well-posedness 
in $H^s(\T)$ for $-\frac{1}{12} < s < 0$. 
These results are motivated by (a) the well-posedness theory of nonlinear dispersive equations with low regularity initial conditions and (b) construction of measures on phase spaces which are invariant under the \eqref{NLS1} evolution.

\subsection{Low Regularity Well-Posedness Theory}

The well-posedness theory for the Cauchy problem \eqref{NLS1} for rough data has been the subject of recent studies. In particular, detailed studies of \eqref{NLS1} have revealed diverse phenomena of the associated data-to-solution map leading to ramified notions of ill-posedness and well-posedness. It is known that: 
\begin{itemize}
	\item The data-to-solution map $H^s\ni u_0 \longmapsto u(t) \in H^s$ (for some $t \neq 0$) is well-defined and analytic provided $s \geq 0$ \cite{Tsutsumi:1987p799} \cite{Bourgain:1993p453}. 

	\item Uniform continuity of the data-to-solution map from $H^s$ to $H^s$ fails for $s<0$ \cite{Kenig:2001p1478, Burq:2002p911, Christ:2003p838}. Moreover, when $s<0$, the data-to-solution map is discontinuous from $H^s (\T)$ even to the space of distributions $(C^\infty (\T))^*$ \cite{Christ:2003p1180, Molinet:2009p365}. 
	\item The data-to-solution map is unbounded from $H^s ( \R )$ to $H^s ( \R )$ provided $s < -\frac{1}{2}$. For example, the norm inflation phenomena identified in \cite{Christ:2003p838} shows there exist initial data arbitrarily small in $H^s(\R)$ which evolve into solutions which are arbitrarily large in $H^s(\R)$ in an arbitrarily short time. 
	\item The data-to-solution map is bounded{\footnote{M. Christ (with J. Holmer and D. Tataru) announced similar results on $\T$ in April 2009 at IHP in Paris.}} from $H^s (\R)$ to $H^s(\R)$ provided $-\frac{1}{6} \leq s <0$ \cite{Koch:2007p782}. Moreover, there exist weak solutions associated to every $u_0 \in H^s (\R)$ in this range. These weak solutions are not known to be unique. 
\end{itemize}
It is unknown whether well-posedness with merely continuous dependence upon the initial data for \eqref{NLS1} holds true in $H^s$ for $s \geq -\frac{1}{2}$. In contrast to these negative results, this paper establishes positive results on subsets of $H^s (\T)$ for certain $s<0$ which are full with respect to natural Gaussian measures.

\subsection{Invariant Gibbs Measures}

Inspired by \cite{Lebowitz:1988p737} and following an approach from \cite{Zhidkov:1994p834}, Bourgain \cite{Bourgain:1994p435}
constructed  the Gibbs measure for{\footnote{In fact, the construction and invariance of the Gibbs measure is proved for a family of (sub-)quintic NLS equations containing \eqref{NLS1} in \cite{Bourgain:1994p435}.}} \eqref{NLS1} and established its invariance under the \eqref{NLS1} flow. Sufficiently regular solutions of \eqref{NLS1} satisfy mass conservation 
\begin{equation}
	\label{mass} \| u(t ) \|_{L^2(\T)} = \| u_0 \|_{L^2 (\T)}, 
\end{equation}

\noi and Hamiltonian conservation 
\begin{equation}
	\label{Hamiltonian} H[u(t) ] = \int_\T \frac{1}{2} |u_x (t) |^2 \pm \frac{1}{4} |u(t)|^4 dx = H[u_0]. 
\end{equation}

\noi By the Hamiltonian structure of the equation, the Gibbs measure 
\begin{equation}
	\label{GibbsMeasure} \text{``}d \mu = e^{-H[u]} \prod_{x \in \T} du(x)\text{''} 
\end{equation}
is formally invariant. The Gibbs measure is rewritten as a weighted Wiener measure 
\begin{equation}
	\label{WeightedWiener} d \mu = Z^{-1} e^{\mp \frac{1}{4} \int|u|^4 dx} d\rho 
\end{equation}
where 
\begin{equation}
	\label{Wiener} d \rho = Z_0^{-1} e^{-\frac{1}{2} \int |u_x|^2 dx} \prod_{x \in \T} du(x) 
\end{equation}
is the Wiener measure on $\T$.

The construction of the Gibbs measure proceeds by showing that the density $ e^{\mp \frac{1}{4} \int|u|^4 dx}$ is in $ L^1 (d \rho)$. Expressed in terms of Fourier coefficients, the Wiener measure describes a Gaussian distribution for each $|n| \widehat{u} (n)$. Thus, a typical element in the support of the Wiener measure may be represented{\footnote{There is an issue regarding the zero Fourier mode which the reader is invited to ignore. The Wiener measure will soon be adjusted using the conserved $L^2$ norm into another formally invariant Gaussian measure which avoids the $n=0$ issue.}} 
\begin{equation}
	\label{representation} u = u^\omega = \sum_{n \in \Z} \frac{g_n (\omega)}{|n|} e^{i n x} 
\end{equation}
where the $\{ g_n \}_{n \in \mathbb{Z}}$ are independent standard complex valued Gaussian random variables
on some probability space $(\Omega, \mathcal{F}, \mathbb{P})$.
Almost surely in $\omega$, the series \eqref{representation} defines a function $u^\omega \in H^{\frac{1}{2}-} (\T)$. Thus, $\int |u|^4 dx $ is well-defined and the density $e^{\mp \frac{1}{4} \int|u|^4 dx} $ may be shown{\footnote{In the defocusing case, this step is clear. The focusing case requires a more delicate analysis exploiting an (invariant) $L^2(\T)$ size cutoff (See \cite{Lebowitz:1988p737} and \cite{Bourgain:1994p435}).}} to be in $L^1 ( \omega).$

The invariance of the Gibbs measure is established by studying a sequence of finite dimensional approximations obtained by Dirichlet-projecting the dynamics of \eqref{NLS1} onto finitely many modes using the fact that the \eqref{NLS1} evolution is well-defined on the support of the Wiener measure. Recall that the evolution for \eqref{NLS1} is well-defined for all $u_0 \in L^2 (\T)$ so it is certainly well-defined on the support of the Gibbs measure living in $H^{\frac{1}{2}-} (\T)$.

The questions of existence and invariance of the Gibbs measure associated to \eqref{NLS1} (in fact, associated to the Wick ordered version \eqref{NLS2}) posed on the two-dimensional torus $\T^2$ were investigated in \cite{Bourgain:1996p446}. In the 
two-dimensional case, the representation \eqref{representation} almost surely in $\omega$ defines a distribution in $H^{0-} (\T^2)$ but not in $L^2 (\T^2)$. More precisely, $u$ defined in \eqref{representation} is almost surely in $B^0_{2, \infty}(\T^2) \setminus L^2(\T^2)$. Since the data-to-solution map is not well-defined on even $L^2 (\T^2)$, 
the issue of well-defined dynamics on the support of the Gibbs measure is not at all obvious. 
Nonetheless, Bourgain \cite{Bourgain:1996p446} established  a well-defined local-in-time dynamics on the support of the Wiener measure. In the defocusing case, he proved global well-posedness almost surely on the support, exploiting the invariance of the (finite dimensional) Gibbs measure.

\subsection{Almost Sure Local Well-Posedness} Consider the canonical Gaussian measure on $H^\alpha (\T)$: 
\begin{equation} \label{Gaussian0}
d{\wt{\rho}}_\alpha = {\wt{Z}}_\alpha^{-1} e^{-\frac{1}{2} \int |D^\alpha u |^2 dx } \prod_{x \in \T} du(x), 
\end{equation}

\noi
where $D = \sqrt{- \partial_x^2}$. The Gaussian measure $d \rho_\alpha$ corresponds to a collection of Gaussian distributions of $\{ |n|^\alpha \widehat{u} (n)\}_{n \in \Z}$, so a typical element in the support may be represented{\footnote{The issue with the zero mode should be ignored; see \eqref{Gaussian1} below.}} as a random Fourier series 
\begin{equation}
	\label{representationalpha} u = u^\omega = \sum_{n \in \Z} \frac{g_n (\omega)}{|n|^\alpha} e^{inx}. 
\end{equation}
This series almost surely in $\omega$ defines a function in $H^{\alpha - \frac{1}{2} -} (\T)$ but not in $H^{\alpha - \frac{1}{2} } (\T)$. Note that $u_0^\omega$ in \eqref{representationalpha} can also be expressed as $u_0^\omega = \sum \wt{g}_n e_n$ where $e_n$ is another orthonormal basis in $H^\al(\mathbb{T})$ and $\{\wt{g}_n\}$ is another family of independent standard complex-valued Gaussian random variables. In this respect, the Gaussian measure $\wt{\rho}_\alpha$ is canonical. See \cite{Kuo:1975p724} for discussions on the Gaussian measures on Banach spaces. Also, see \cite{Zhidkov:2001p831}.

Since $\| u(t) \|_{L^2} = \|u_0 \|_{L^2}$ under the flow of \eqref{NLS1}, we formally expect the Gaussian measure on $L^2(\T)$ 
\begin{equation}
	\label{eq:white} d\rho_0 = Z_0^{-1} e^{-\frac{1}{2} \int | u |^2 dx } \prod_{x \in \T} du(x) 
\end{equation}

\noi to be invariant in view of the Hamiltonian structure of \eqref{NLS1}. This measure $\rho_0$ is the white noise on the distributions on $\T$ and is supported on $H^{-\frac{1}{2}-} (\T) \setminus H^{-\frac{1}{2}} (\T)$, i.e. in the scaling critical/supercritical regime for \eqref{NLS1}. It was shown in \cite{Oh:2010p1338} that the white noise $\rho_0$ is a weak limit of the invariant measures under the flow of \eqref{NLS1}. However, this result does not establish the invariance of the white noise $\rho_0$ since the flow is not well-defined on its support. (See Remark \ref{REM:white}.) Invariance of white noise has recently been established for the KdV equation on $\T$ \cite{Quastel:2008p796, Oh:2009p792, Oh:2010p1338}. See \cite{Oh:2009p1296} for a summary of these results.

If we define $v(t) = e^{i \gamma t} u(t)$, with $\gamma \in \R$, where $u$ solves \eqref{NLS1}, then $v$ satisfies $i 
\partial_t v - v_{xx} \pm |v|^2v + \gamma v = 0$. Recall that $\avint |u|^2 dx := \frac{1}{2\pi} \int |u|^2 dx$ is conserved under the flow of \eqref{NLS1} for $u_0 \in L^2 (\T)$. Hence, by letting $\gamma = \mp 2 \fint |u|^2 dx$, \eqref{NLS1} is equivalent to 
\begin{equation}
	\label{NLS2} 
	\begin{cases}
		i u_t - u_{xx} \pm (u |u|^2 -2 u \fint |u|^2 dx) = 0 \\
		u|_{t= 0} = u_0, 
	\end{cases}
\end{equation}

\noi at least for $u_0 \in L^2(\T)$. However, for $u_0 \notin L^2(\T)$, we can't freely convert solutions of \eqref{NLS2} into solutions of \eqref{NLS1}. Bourgain \cite{Bourgain:1996p446} refers to \eqref{NLS2} as 
the {\it Wick ordered cubic NLS} since it may also be obtained from the Wick ordered Hamiltonian. 

In the following, we choose to study \eqref{NLS2} instead of \eqref{NLS1} for $u_0 \notin L^2 (\T)$. (See Remark \ref{REM:renorm}.) In particular, we consider $u_0 $ of the form (slightly adjusted compared with \eqref{representationalpha}) 
\begin{equation}
	\label{IV} u_0 = u_0^\omega = \sum_{n \in \Z} \frac{g_n (\omega)}{\sqrt{1+ |n|^{2\alpha}}} e^{inx} 
\end{equation}

\noi which can be regarded as a  typical element in the support of the Gaussian measure 
\begin{equation}
	\label{Gaussian1} d \rho_\al = Z_\alpha^{-1} \exp \Big(-\frac{1}{2}\int |u|^2 dx -\frac{1}{2} \int |D^{\al} u|^2 dx\Big) \prod_{x \in \mathbb{T}} d u(x). 
\end{equation}

\noi By shifting the Laplacian as in \cite{Bourgain:1994p435, Bourgain:1996p446}, i.e. replacing $-u_{xx}$ by $-u_{xx} +u$ in \eqref{NLS1} or \eqref{NLS2}, we can also regard $u_0$ of the form \eqref{IV} as the functions in the support of the Gaussian measure $ \wt{\rho}_\alpha$ defined in \eqref{Gaussian0}. 
(Strictly speaking, one needs to replace the denominator in \eqref{IV} by $(1+|n|^2)^\frac{\al}{2}$
in this case.)
Note that $u_0^\omega$ in \eqref{IV} is in $\bigcap_{s < \al - \frac{1}{2}} H^s \setminus H^{ \al - \frac{1}{2}}$.
In view of Bourgain's global well-posedness (GWP) result in $L^2(\mathbb{T})$ in \cite{Bourgain:1993p453}, 
we assume that $\al \leq \frac{1}{2}$ in the following 
so that $u_0^\omega$ lies strictly in the negative Sobolev spaces, almost surely in  $\omega$.

In establishing local well-posedness, we follow the argument by Bourgain \cite{Bourgain:1996p446}. 
First, write \eqref{NLS2} as an integral equation as in \eqref{NLS3}.
\begin{equation} 
	\label{NLS3} u(t) = \G u(t) := S(t) u_0 \pm i \int_0^t S(t - t') \mathcal{N}(u) (t') d t' 
\end{equation}

\noi where $S(t) = e^{-i \dx^2 t}$, $u_0$ is as in \eqref{IV}, and \[\mathcal{N}(u) := u |u|^2 - 2u \fint |u|^2.\] 

\noi
Note that $S(t) u_0$ has the same regularity as $u_0$ for each fixed $t \in \mathbb{R}$. i.e. $S(t) u^\omega_0 \in H^{\al-\frac{1}{2}-} (\mathbb{T})\setminus H^{\al-\frac{1}{2}}(\mathbb{T})$ a.s. Hence, $S(t)u_0$ is strictly in the negative Sobolev space for $\al \leq \frac{1}{2}$ a.s.

However, it turns out that 
the nonlinear part $\int_0^t S(t - t') \mathcal{N}(u) (t') d t'$ lies 
almost surely in a smoother space $L^2 (\mathbb{T})$ even for $\al \leq \frac{1}{2}$. 
(Also, see \cite{Bourgain:1996p446}, \cite{Burq:2008p624}.) 
We indeed show that for each small $\dl> 0$ there exists $\Omega_\dl$ 
with complemental measure $< e^{-\frac{1}{\dl^c}}$ such that $\G$ defined in \eqref{NLS3} 
is a contraction on $S(t) u_0^\omega + B$ for $\omega \in \Omega_\dl$ on the time interval $[0, \dl]$, 
where $B$ denotes the ball of radius 1 in the Bourgain space $X^{s, \frac{1}{2}+, \dl}$ for some $s \geq 0$.
(See \eqref{Xsb} and \eqref{Xsb2} for the definition of $X^{s, \frac{1}{2}+, \dl}$.)

 The following theorem states almost sure local well-posedness for each $\al \in (\frac{1}{6}, \frac{1}{2}].$
\begin{maintheorem}
	\label{THM:LWP} Let $\al \in ( \max( \frac{s}{3} + \frac{1}{6}, s), \frac{1}{2}]$ with $ s \in [0, \frac{1}{2}]$. Then, the periodic (Wick ordered) cubic NLS \eqref{NLS2} is locally well-posed almost surely in $H^{\al - \frac{1}{2}-}(\mathbb{T})$. More precisely, there exist $c > 0$ such that for each $\dl \ll 1$, there exists a set $\Omega_\dl \in \mathcal{F}$ with the following properties:
	\begin{enumerate}
		\item[(i)] 
		The complemental measure of $\Omega_\dl$ is small. More precisely, we have
		\[\mathbb{P}(\Omega_\dl^c) = \rho_\al \circ u_0(\Omega_\dl^c) < e^{-\frac{1}{\dl^c}},\]

		\noi 
		where $\rho_\al$ is the Gaussian probability measure on $H^{\al-\frac{1}{2}-}(\T)$ defined in \eqref{Gaussian1}
		and $u_0$ is viewed as a map $u_0:\Omega \to H^{\al-\frac{1}{2}-}(\mathbb{T})$.
		
		\item[(ii)] For each $\omega \in \Omega_\dl$, there exists a (unique) solution $u$ of \eqref{NLS2} in
		\[e^{-i \dx^2 t}u_0 + C([-\dl, \dl];H^{s}(\mathbb{T})) \subset C([-\dl, \dl];H^{\al - \frac{1}{2}-}(\mathbb{T}))\]
		with the initial condition $u_0^\omega$ given by \eqref{IV}. 
		Here, the uniqueness holds only in the ball centered at $e^{-i \dx^2 t}u_0$ of radius 1 in $X^{s, \frac{1}{2}+,\dl}$.
	\end{enumerate}
	
	\noi In particular, we have almost sure local well-posedness with respect to the Gaussian measure \eqref{Gaussian1} supported in $H^{\s}(\mathbb{T})$ for each $\s > -\frac{1}{3}$. 
\end{maintheorem}

\noi
We prove Theorem \ref{THM:LWP} in Section \ref{SEC:LWP}
by a combination of deterministic multilinear estimates (e.g. Lemma \ref{LEM:deterministic})
and probabilistic estimates on the linear solution with random initial data (Lemmata \ref{LEM:prob1}, \ref{LEM:prob2}, and \ref{LEM:hyper}.)

\subsection{Almost Sure Global Well-Posedness} \label{SUBSEC:1GWP}
We continue our study  on the periodic cubic NLS \eqref{NLS1} with random initial data in the negative Sobolev spaces. 
In the second part of this paper, we study global well-posedness of  \eqref{NLS1} 
with initial data of the form \eqref{IV}.
In particular, we establish almost sure global well-posedness of  \eqref{NLS1} 
with respect to the Gaussian measure $\rho_\al$ in \eqref{Gaussian1}
for certain values of $\al \leq \frac{1}{2}$.

So far, there is basically only one method known for proving  almost sure global well-posedness of PDEs 
with random initial data of type \eqref{IV}. 
In \cite{Bourgain:1994p435}, Bourgain proved the invariance of the Gibbs measures for NLS. In dealing with the super-cubic nonlinearity, (where only the local well-posedness result was available), he used a probabilistic argument and the approximating finite dimensional ODEs (with the invariant finite dimensional Gibbs measures) to extend the local solutions to global ones almost surely on the statistical ensemble and showed the invariance of the Gibbs measures. We point out that this method can be applied in a general setting, provided that  local well-posedness is obtained with a ``good" estimate on the solutions (e.g. via the fixed point argument) and that we have a formally invariant measure such as the Gibbs measure or the white noise (where the leading term corresponds to \eqref{Gaussian1} for $\al = 1$ and $\al = 0$.) See Bourgain \cite{Bourgain:1994p540, Bourgain:1996p446}, Burq-Tzvetkov \cite{Burq:2007p1542,Burq:2008p623}, Oh \cite{Oh:2009p791, OhSBO, Oh:2009p1296}, and Tzvetkov \cite{Tzvetkov:2006p801,Tzvetkov:2008p736}.

From Theorem \ref{THM:LWP}, we have local solutions in the support of the Gaussian measure $\rho_\al$ in \eqref{Gaussian1} for $\al \in (\frac{1}{6}, \frac{1}{2}]$, which we would like to extend globally in time. 
Since the values of $\al$ is strictly between 0 and 1, the initial condition $u_0$ in \eqref{IV} is not in the support of an invariant measure for \eqref{NLS2} i.e. 
$\rho_\al$ in \eqref{Gaussian1} does not correspond to (the quadratic part of) the Gibbs measure or the white noise.
 Therefore, Bourgain's probabilistic argument \cite{Bourgain:1994p435} is {\it not} applicable here.

The crucial point in the local theory is the fact 
that the nonlinear part is almost surely smoother than the initial data.
This observation led us to consider Bourgain's high-low method \cite{Bourgain:1998p434}
for establishing global well-posedness, 
 since this kind of {\it nonlinear smoothing} is the crucial ingredient for the method. Moreover, as you see below, the implementation of the high-low method naturally lets us apply our probabilistic local theory iteratively since the data for the difference equations with high frequency initial data 
 have random Fourier coefficients (with the same distribution) at each step.

\medskip
In the following, we briefly sketch the iteration scheme for global well-posedness.
Let $s = \al - \frac{1}{2}-$ with $\al \leq \frac{1}{2}$. i.e. $s<0$.\footnote{In the global theory, 
we use $s = \al - \frac{1}{2}- < 0$ to denote
the regularity of the initial data below $L^2$.}

 By the large deviation estimate, we have 
\begin{equation}
	\label{largedevi} \mathbb{P} ( \| u_0(\omega)\|_{H^s} \geq K ) \leq e^{-c K^2}. 
\end{equation}

\noi In the following, we restrict ourselves on $\Omega_K = \{ \omega \in \Omega: \| u_0(\omega)\|_{H^s} \leq K \}.$ By writing $u_0 = \phi_0 + \psi_0$, where $\phi_0 := \mathbb{P}_{\leq N} u_0 =\sum_{|n|\leq N} \ft{u}_0(n) e^{inx}$,
the low-frequency part $\phi_0$ is in $L^2(\mathbb{T})$, and it satisfies 
\[\| \phi_0\|_{L^2} \leq N^{-s} \|\phi_0\|_{H^s} \leq N^{-s} K.\] 

Let $u^{1}$ denote the solution of \eqref{NLS2} with the initial data $\phi_0$ on some time interval $[0, \dl]$, where $\dl$ is the time of local existence, i.e. $\dl = \dl(N^{-s}K) \lesssim \dl (\|\phi_0\|_{L^2})$. Then, we have 
\begin{equation}
	\label{LNLS1} 
	\begin{cases}
		i \dt u^{1} - \dx^2 u^{1} \pm \mathcal{N}(u^{1}) = 0 \\
		u^{1}|_{t= 0} = \phi_0. 
	\end{cases}
\end{equation}

\noi From the $L^2$ well-posedness theory of Bourgain \cite{ Bourgain:1993p453}, \eqref{LNLS1} is globally well-posed with the $L^2$-conservation: $\|u^{1}(t)\|_{L^2} = \|\phi_0\|_{L^2} \lesssim N^{-s}K $ for any $t \in \R$. Moreover, from the local theory, we have 
\begin{equation}
	\label{u1bound} \|u^1\|_{X^{0, \frac{1}{2}+}[0, \dl]} \lesssim \|\phi_0\|_{L^2} \leq N^{-s}K. 
\end{equation}

Now, let $v^{1}$ be a solution of the following difference equation on $[0, \dl]$: 
\begin{equation}
	\label{HNLS1} 
	\begin{cases}
		i \dt v^{1} - \dx^2 v^{1} \pm (\mathcal{N} (u^{1} + v^{1}) - \mathcal{N}(u^{1})) = 0 \\
		v^{1}|_{t= 0} = \psi_0 = \sum_{|n|> N} \frac{g_n(\omega)}{\sqrt{1+|n|^{2\al}}} e^{inx}. 
	\end{cases}
\end{equation}

\noi i.e. we have $u(t) = u^{1}(t) + v^{1}(t)$ as long as the solution $v^{1}$ of \eqref{HNLS1} exists.
Note that $\psi_0$ has Gaussian-randomized Fourier coefficients. Hence, we can use our probabilistic local theory 
(as in Theorem \ref{THM:LWP})
to study \eqref{HNLS1}.

Suppose that, by our probabilistic local theory, we can show that \eqref{HNLS1} is locally well-posed on the time interval $[0, \delta]$ except on a set of measure $e^{-\frac{1}{\dl^c}}$. We have $v^1(t) = S(t) \psi_0 + w^1(t)$, where the nonlinear part $w^1(t)$ is smoother and is in $L^2(\T)$ for all $t \in [0, \dl]$. The appearance of the external function $u^1$ in \eqref{HNLS1} with large $X^{0, \frac{1}{2}+}_{[0,  \dl]}$-norm, forces us to refine our argument used to prove Theorem \ref{THM:LWP} to obtain a good estimate on $\|w^1(t)\|_{L^2}$.

At time $t = \dl$, we redistribute the data. i.e. write $u (\dl) = \phi_1 + \psi_1$, where $\phi_1 := u^1(\dl) + w^1(\dl)$ and $\psi_1 := S(\dl) \psi_0$. Let $u^2$ denote the solution of \eqref{NLS2} with the initial data $\phi_1$ starting at time $t = \dl$. i.e. 
\begin{equation}
	\label{LNLS2} 
	\begin{cases}
		i \dt u^2 - \dx^2 u^2 \pm \mathcal{N}(u^2) = 0 \\
		u^2|_{t= \dl} = \phi_1 = u^1(\dl) + w^1(\dl) \in L^2(\T). 
	\end{cases}
\end{equation}

\noi Then, \eqref{LNLS2} is globally well-posed. Also, from the local theory, we have 
\begin{equation}
	\label{u2bound} \| u^2 \|_{X^{0, \frac{1}{2}+}[\dl, 2\dl]} \lesssim \|\phi_1\|_{L^2} \leq \| u^1(\dl)\|_{L^2} + \|w^1(\dl)\|_{L^2} \lesssim N^{-s}K + \|w^1(\dl)\|_{L^2} \lesssim N^{-s} K 
\end{equation}

\noi {\it as long as} 
\begin{equation}
	\label{w1bound} \|w^1(\dl)\|_{L^2} \lesssim N^{-s}K. 
\end{equation}

Now, let $v^2$ be the solution of the difference equation on $[\dl, 2\dl]$: 
\begin{equation}
	\label{HNLS2} 
	\begin{cases}
		i \dt v^2 - \dx^2 v^2 \pm (\mathcal{N} (u^2 + v^2) - \mathcal{N}(u^2)) = 0 \\
		v^2|_{t= \dl} = \psi_1 = \sum_{|n|> N} \frac{g_n(\omega)e^{i \dl n^2}}{\sqrt{1+|n|^{2\al}}} e^{inx}. 
	\end{cases}
\end{equation}

\noi Once again, $\psi_1$ has Gaussian-randomized Fourier coefficients.
Since the complex Gaussian is invariant under rotation, 
we see that $\psi_1$ has the same distribution as $\psi_0$.\footnote{This can be viewed as invariance
of the Gaussian measure $\rho_\al$ (restricted to the high frequencies) under the linear flow.
This is the key ingredient for the global-in-time argument (in the absence of (formally) invariant measures under the nonlinear PDE flow.)} 
Hence, we can use our probabilistic local theory to study \eqref{HNLS2}.

In this way, we iterate the deterministic local theory to the ``low-frequency'' part $u^j$ and the probabilistic local theory to the ``high-frequency'' part $v^j$ to prove that \eqref{NLS2} is well-posed on $[0, T]$ for arbitrary $T>0$. 
For details, see Section \ref{SEC:GWP}.
\begin{maintheorem}
	\label{THM:GWP1} Let $\al \in ( \frac{5}{12}, \frac{1}{2}]$. Then, the periodic (Wick ordered) cubic NLS \eqref{NLS2} is globally well-posed almost surely in $H^{\al - \frac{1}{2}-}(\mathbb{T})$. More precisely, for almost every $\omega \in \Omega$ there exists a (unique) solution $u$ of \eqref{NLS2} in
	\[e^{-i \dx^2 t}u_0 + C(\R;L^2(\mathbb{T})) \subset C(\R;H^{\al - \frac{1}{2}-}(\mathbb{T}))\]
	with the initial condition $u_0^\omega$ given by \eqref{IV}.
	Here, the uniqueness holds in a very mild sense. See Remark \ref{REM:unique}.

	In particular, we have almost sure global well-posedness with respect to the Gaussian measure \eqref{Gaussian1} supported in $H^{s}(\mathbb{T})$ for each $s > -\frac{1}{12}$. 
\end{maintheorem}

\subsection{Remarks}

We conclude this introduction by stating several important remarks.
\begin{remark} \label{REM:ABS}
	\rm A linear part of a local-in-time solution constructed in Theorem \ref{THM:LWP} indeed lies in $C([-\dl, \dl];B(\mathbb{T}))$ for any Banach space $B (\T) \supset H^\al(\T)$ such that $(H^\al, B, \rho_\al)$ is an abstract Wiener space. (Roughly speaking, 
	an abstract Wiener space is a Banach space extension $B(\T)$ of  $H^\al(\T)$, where the Gaussian measure $\rho_\al$ makes sense as a countable additive probability measure.) In this case, a solution $u$ to \eqref{NLS2} lies in
	\[ u = e^{-i \dx^2 t}u_0 + (-i\dt + \dx^2)^{-1} u \in C([-\dl, \dl];B(\mathbb{T})) + C([-\dl, \dl];H^{s}(\mathbb{T}))\]
	
\noi	
for some $s\geq 0$ as in Theorem \ref{THM:LWP}. 	As examples of $B$, we can take the Sobolev spaces $W^{\s, p}$ with $ \s < \al - \frac{1}{2}$, the Fourier-Lebesgue spaces $\mathcal{F}L^{\s, p}$ with $\s < \al - \frac{1}{p}$, where $\mathcal{F}L^{\s, p}$ is defined via the norm $\|f\|_{\mathcal{F}L^{\s, p}} = \|\jb{n}^\s \ft{f}(n)\|_{L^p_n}$,
	and  the Besov spaces $B^{\al - \frac{1}{2}}_{p, \infty}$ with $p < \infty$. 
	See B\'enyi-Oh \cite{Benyi:2010p842} for 	regularity of $\rho_\al$ (and $u_0$ in \eqref{IV}) in different function spaces. 
	In \cite{Benyi:2010p842}, we study the regularity of $\rho_\al$ for $\al = 1$ but it can be easily adjusted for any $\al$.
A similar comment applies to global-in-time solutions constructed in Theorem \ref{THM:GWP1}.
For global-in-time argument, however, it is important that the large deviation estimate \eqref{largedevi} still holds for these spaces.
\end{remark}

\begin{remark} \label{REM:unique}
	\rm In the local theory of Theorem \ref{THM:LWP}, 
	uniqueness holds only in the ball centered at $S(t) u_0^\omega$ of radius 1 in $X^{s, \frac{1}{2}+,\dl}$
	for some $s\geq 0$. Continuous dependence on the initial data holds, in some weak sense,  in $H^s(\T)$ for some $s \geq 0.$  
	(See Subsection \ref{SUBSEC:LWP1}.)
	Also, note that Theorem \ref{THM:LWP} can not be applied to \eqref{NLS1}, since $u_0^\omega$ is almost surely not in $L^2(\mathbb{T})$.

In the global theory of Theorem \ref{THM:GWP1}, 
the situation is a little more complicated.
 On the one hand, uniqueness and continuous dependence 
 for ``low-frequency'' part $u^j$ in the $j$th step hold in $C([(j-1) \dl, j\dl], L^2(\T)) \cap X^{0, \frac{1}{2}+} [(j-1) \dl, j\dl]$ as usual. On the other hand, uniqueness for the high-frequency part $v^j$ in the $j$th step holds only in the ball centered at $S(t) \psi_{j-1}$ of small radius in $X^{0, \frac{1}{2}+} [(j-1) \dl, j\dl]$. 
 Also, weak continuous dependence for $v^j$ holds in $L^2(\T)$ in the sense
 analogous to the local theory in Theorem \ref{THM:LWP}. 
  \end{remark}

\begin{remark} \label{REM:white}
\rm

Recall that the white noise corresponds to $\al = 0$ in \eqref{Gaussian1} (up to constants).
Hence,  Theorems \ref{THM:LWP} and \ref{THM:GWP1} may also be viewed as partial results towards showing 
well-posedness of \eqref{NLS2} on the support of the white noise $ \rho_0$. 
\end{remark}

\begin{remark}
	\rm The periodic cubic NLS \eqref{NLS1} is known to be ill-posed in $H^s(\mathbb{T})$ for $s < 0$. See Molinet \cite{Molinet:2009p365} for the most recent work and the references therein. 
	As for the Wick ordered cubic NLS \eqref{NLS2}, note that $ u_{N, a}(x, t) = a e^{i(Nx + N^2 t \mp |a|^2t)}$ solves the Wick ordered cubic NLS \eqref{NLS2} for $a \in \mathbb{C}$ and $N \in \mathbb{N}$. Hence, by following the argument of Burq-G\'erard-Tzvetkov \cite{Burq:2002p911}, we can show failure of uniform continuity of the solution map of \eqref{NLS2} below $L^2(\mathbb{T})$. Thus, it is nontrivial to construct solutions of \eqref{NLS2} in the negative Sobolev spaces.
Also, see Christ-Colliander-Tao \cite{Christ:2003p838}. 	
	
As mentioned earlier, Molinet \cite{Molinet:2009p365} showed that \eqref{NLS1} is not well-posedness below $L^2(\mathbb{T})$ by proving the weak discontinuity of the flow map in $L^2(\T)$. We point out that his argument does not apply to \eqref{NLS2}. 
Indeed, it is shown in \cite{SULEM} that the solution map to the Wick ordered cubic NLS \eqref{NLS2} is
weakly continuous in $L^2(\T)$.
\end{remark}

\begin{remark}
	\label{REM:FLP} \rm
	On the one hand, it is known that $u_0^\omega$ of the form \eqref{IV} is in $\mathcal{F}L^{s, p}$ almost surely for $ s < \al - \frac{1}{p}$ and not in the smoother spaces. See \cite{Oh:2009p792, Benyi:2010p842}. On the other hand, Christ \cite{Christ:2007p1011} constructed local-in-time solutions in $\mathcal{F} L^{0,p}$ for $2< p < \infty$ by the power series method. Also see Gr\"unrock-Herr \cite{Grunrock:2008p659} for the same result via the fixed point argument. Hence, it follows from their result that \eqref{NLS2} with $u_0^\omega$ in \eqref{IV} is almost surely locally well-posed for $\al > 0$, but the solution $u$ lies in $C([-\dl, \dl]; \mathcal{F} L^{0,\frac{1}{\al}+}(\T))$. 
	
	In the following, 
	we first construct local-in-time solutions in $C([-\dl, \dl];H^{\al - \frac{1}{2}-}(\mathbb{T}))$ by exhibiting nonlinear smoothing under randomization. Also, see Remark \ref{REM:ABS}.
	In Theorem \ref{THM:GWP1},  we extend the local solutions to global ones (in the absence of invariant measures) by exploiting such nonlinear smoothing. 
\end{remark}

\begin{remark}
	\label{REM:renorm}\rm In \cite{Bourgain:1996p446}, the two dimensional Wick ordered (defocusing) cubic NLS appeared as an equivalent formulation of (the limit of the finite dimensional) Hamiltonian equation, arising from the Wick ordered Hamiltonian. Such renormalization on the nonlinearity was a natural consequence of the Euclidean $\varphi_2^4$ quantum field theory. In our case, by taking the initial data $u_0^\omega$ to be of the form \eqref{IV} with $\al \leq \frac{1}{2}$, \eqref{NLS2} also arises as an equivalent formulation of (the limit of the finite dimensional) Hamiltonian equation from the Wick ordered Hamiltonian, (at least for $\al > \frac{1}{4}$) under Gaussian assumption on solutions. Moreover, such renormalization is needed to obtain the continuous dependence on the initial data \cite{Christ:2007p1011, Grunrock:2008p659}.
See \cite{SULEM} for more discussion on this issue.
\end{remark}

\begin{remark}
	\rm In \cite{Bourgain:1996p446}, local solutions were constructed via the fixed point argument around the linear solution $z_1(t) := S(t)u_0$ with probabilistic arguments. Also see Burq-Tzvetkov \cite{Burq:2008p624, Burq:2008p623} and Thomann \cite{Thomann:2009p1427} for related arguments. While the basic probabilistic argument (e.g. Lemma \ref{LEM:prob2}) is similar, the argument in \cite{Burq:2008p624, Thomann:2009p1427} further exploits the properties of the eigenfunctions, and the argument in \cite{Bourgain:1996p446} and this paper exploits more properties of the product of Gaussian random variables via the hypercontractivity of the Ornstein-Uhlenbeck semigroup. 
(See Lemma \ref{LEM:hyper}.)

In \cite{OhFE}, the second author considered KdV and dispersionless Szeg\"o equation
with random initial data.
Even with random initial data, 
an attempt to construct local-in-time solutions by the fixed point argument
around the linear solution (below the deterministic threshold) failed for both of these equations. 
Nonetheless, local-in-time solutions for KdV 
(below the deterministic threshold) were constructed
via the second iteration argument, exploiting randomization of initial data.
\end{remark}

\medskip

This paper is organized as follows. 
In Section 2, we introduce the basic function spaces and notations. 
In Section 3, we list some deterministic and probabilistic lemmata. 
Then, we prove Theorem \ref{THM:LWP} in Section 4
and Theorem \ref{THM:GWP1} in Section 5.

\section{Notation} First, recall the Bourgain space $X^{s, b}(\mathbb{T} \times \mathbb{R})$, c.f. \cite{Bourgain:1993p453}, whose norm is given by
\begin{equation} \label{Xsb}
\| u\|_{X^{s, b}(\mathbb{T} \times \mathbb{R})} = \|\jb{n}^s \jb{\tau - n^2}^b \ft{u}(n, \tau)\|_{l^2_n L^2_\tau } 
\end{equation}

\noi where $\jb{\, \cdot\, } = 1 + |\cdot|$.
Recall that $X^{s, b}$ embeds into $C_t H^s_x$ for $b  >\frac{1}{2}$.
%
%
%
 We also define the local-in-time version $X^{s, b, \dl }$ on $\T \times [-\dl, \dl]$, by
\begin{equation} \label{Xsb2}
 \|u\|_{X^{s, b, \dl }} = \inf \big\{ \|\wt{u} \|_{X^{s, b}(\T \times \mathbb{R})}: {\wt{u}|_{[-\dl, \dl]} = u}\big\}.
\end{equation}

\noi 
We also define the local-in-time version $X^{s, b}_I = X^{s, b}[a, b]$
on an interval $I = [a, b]$.
The local-in-time versions of other function spaces are defined analogously.

For simplicity, we often drop $2\pi$ in dealing with the Fourier transforms. 
If a function $f$ is random, we may use the superscript $f^\omega$ to show the dependence on $\omega$.

We use $\eta \in C^\infty_c(\mathbb{R})$ to denote a smooth cutoff function supported on $[-2, 2]$ with $\eta \equiv 1$ on $[-1, 1]$ and let $\eta_{_\dl}(t) =\eta(\dl^{-1}t)$,
and $\chi = \chi_{[-1, 1]}$ to denote the characteristic function of the interval $[-1, 1]$
 and let $\chi_{_\dl}(t) =\chi(\dl^{-1}t) = \chi_{[-\dl, \dl]}(t)$.

The decreasing rearrangement of dyadic numbers $N_1, N_2, N_3$ will be denoted $N^1, N^2, N^3$, following \cite{Bourgain:1996p446}.

We use $c,$ $ C$ to denote various constants, usually depending only on $\al$ and $s$. If a constant depends on other quantities, we will make it explicit. We use $A\lesssim B$ to denote an estimate of the form $A\leq CB$. Similarly, we use $A\sim B$ to denote $A\lesssim B$ and $B\lesssim A$ and use $A\ll B$ when there is no general constant $C$ such that $B \leq CA$. We also use $a+$ (and $a-$) to denote $a + \eps$ (and $a - \eps$), respectively, for arbitrarily small $\eps \ll 1$.

\section{Deterministic and Probabilistic Lemmata}


\subsection{Deterministic Lemmata}
First, recall the following algebraic identity related to the cubic NLS: 
\begin{equation}
	\label{ALGEBRA} n^2 - (n_1^2 - n_2^2 + n_3^2) = 2(n_2 - n_1) (n_2 - n_3) 
\end{equation}

\noi for $n = n_1 - n_2 + n_3$. Let $N^1, N^2, N^3$ be the decreasing ordering of $N_1, N_2, N_3$, where $|n_j| \sim N_j$, and let $n^j$ denote the corresponding frequency.

Next, recall the following number theoretic fact \cite{HW}.
Given an integer $m$, let $d(m)$ denote the number of divisors of $m$.
Then, we have
\begin{equation} \label{DIVISOR}
d(m) \lesssim e^{c \frac{\log m}{\log \log m}} \ ( = o(m^\eps) \text{ for any }\eps>0.)
\end{equation}

\noi
From this fact, we obtain the following lemma.

\begin{lemma}
	\label{LEM:count1} 
	Fix $\mu \in \mathbb{Z}.$ Let 
	\begin{align}
		\label{Smu} S_\mu = \{ (n_1, n_2, n_3) \in \mathbb{Z}^3: |n_j|\sim N_j, n_2 \ne n_1, n_3, \textup{ and } 2(n_2 - n_1) (n_2 - n_3) = \mu\}. \notag 
	\end{align}
	
	\noi Then, we have 
	\begin{equation}
		\label{counting} \# S_\mu \lesssim (N^1)^{0+}N^3. 
	\end{equation}
\end{lemma}
\begin{proof}
	By assumption, we have $|\mu| \lesssim (N^1)^2$. Hence, the number of the divisors of $\mu$ is $o((N^1)^\eps)$ for any $\eps > 0$. Without loss of generality, assume $N^3 \sim \min(|n_2|, |n_3|)$.
	
	First, suppose $|n_2|\sim N^3$.  Fix $n_2$. Then, from \eqref{DIVISOR}, there are at most $o((N^1)^{0+})$ many choices for $d := n_2 - n_1$. Then, there are at most $o((N^1)^{0+})$ many choices for $n_1$ and $n_3$ since $n_1 = n_2 - d$ and $n_3 = n_2 - \frac{\mu}{2d}$.
	
	Next, suppose $|n_3|\sim N^3$. Fix $n_3$. Then, from \eqref{DIVISOR}, there are at most $o((N^1)^{0+})$ many choices for $d := n_2 - n_3$. Then, there are at most $o((N^1)^{0+})$ many choices for $n_1$ and $n_2$ since $n_2 = d + n_3$ and $n_1 = d + n_3 - \frac{\mu}{2d}$. Hence, \eqref{counting} holds in both cases. 
\end{proof}

Recall that by restricting the Bourgain spaces onto a small time interval $[-\dl, \dl]$, we can gain a small power of $\dl$ (at a slight loss of regularity on $\jb{\tau - n^2} $.) See \cite{Bourgain:1993p453}.
\begin{lemma}
	\label{LEM:timedecay} Let $s \in\R$ and $b < \frac{1}{2}$. Then, there exists $C = C(b) > 0$ such that  we have 
	\begin{equation}
		\label{CCtime0}
		 \|u\|_{X^{s, b, \dl}}  \leq C \dl^{\frac{1}{2}-b-} \|u\|_{X^{s, \frac{1}{2}, \dl}}. 
	\end{equation}
\end{lemma}

\noi
Before presenting the proof, first recall the following fact from \cite{BOPCMI}.
Let $\chi_{_\dl} (t):= \chi_{[-\dl, \dl]}(t)$ be the characteristic function of the interval $[-\dl, \dl]$.
Then, for $b< \frac{1}{2}$, we have
\begin{equation} \label{CC00}
\|\chi_{_\dl}(t) u \|_{X^{s, b}} \sim \|u\|_{X^{s, b, \dl}}.
\end{equation}

\noi
Indeed, by definition \eqref{Xsb2} of local-in-time $X^{s, b}$, 
we have $ \|u\|_{X^{s, b, \dl}} \leq \|\chi_{_\dl}(t) u \|_{X^{s, b}}$.
The inequality in the other direction:
$ \|\chi_{_\dl}(t) u \|_{X^{s, b}} \leq C(b) \|u\|_{X^{s, b, \dl}} $
follows from the boundedness of multiplication by 
a sharp cutoff function
in $H^b_t$ for $b < \frac{1}{2}$.
Note that the constant $C(b)$ depends only on $b$,
in particular independent of $\dl$.

\begin{proof}
Let $\wt{u}$ be any extension of $u$ onto $\R$,
i.e. $\wt{u}$ is a function on $\R$ such that $\wt{u} = u$ on $[-\dl, \dl]$.
Also, let $v = \chi_{_\dl}(t) \wt{u}$.
Then, we have $v = \wt{u} = u$ on $[-\dl, \dl]$.
Moreover, from \eqref{CC00}, we have
\begin{equation} \label{CC0}
\|v\|_{X^{s, b}} \sim \|\wt{u}\|_{X^{s, b, \dl}} = \|u\|_{X^{s, b, \dl}}
\end{equation}

\noi
for $b < \frac{1}{2}$.

By interpolation, we have 
	\begin{equation}
		\label{CCtime1} 
		\| v\|_{X^{s, b}} \lesssim \| v\|^\alpha_{X^{s, 0 }}
		\| v\|^{1-\alpha}_{X^{s, \frac{1}{2}- }}, 
	\end{equation}
	
\noi 
where $\alpha = 1-(2+)b \in(0, 1)$. Recall 
$\ft{\chi_{_\dl}}(\tau) = \dl \ft{\chi}(\dl\tau)$,
where $\chi = \chi_{[-1, 1]}$. 
Hence, we have 
\begin{equation}
 \|\ft{\chi_{_{\dl}}}\|_{L^q_\tau} \sim \dl^\frac{q-1}{q} \|\ft{\chi}\|_{L^q_\tau} 
\lesssim \dl^\frac{q-1}{q},
\label{decay}
\end{equation} 

\noi
where the last inequality holds for $q > 1$.
Thus, we can gain a positive power of $\dl$ as long as $q>1$. 
For fixed $n$, by Young and H\"older inequalities, we have 
\begin{align*}
    \| \ft{v}(n, \cdot) \|_{L^2_\tau} 
    & = 
    \| \ft{\chi_{_{\dl}}} * \ft{\wt{u}}(n, \cdot) \|_{L^2_\tau} 
		\leq \| \ft{\chi_{_{\dl}}} \|_{L^{2-}_\tau} \|\ft{\wt{u}}(n, \cdot) \|_{L^{1+}_\tau} \\
		& \lesssim \dl^{\frac{1}{2}-} \|\jb{\tau-n^2}^{-\frac{1}{2}} \|_{L^{2+}_\tau} 
		\|\jb{\tau-n^2}^{\frac{1}{2}}\ft{\wt{u}}(n, \cdot) \|_{L^2_\tau}\\
		& \lesssim \dl^{\frac{1}{2}-} 
		\|\jb{\tau-n^2}^{\frac{1}{2}}\ft{\wt{u}}(n, \cdot) \|_{L^2_\tau}.
\end{align*}
	
\noi	
 Hence, for $p > 2$, we have 
\begin{equation}
\label{CCtime2} \| v\|_{X^{s, 0 }} \lesssim \dl^{\frac{1}{2}-} \| \wt{u}\|_{X^{s, \frac{1}{2} }}. 
\end{equation}
	
\noi 
Then, from \eqref{CC0}, \eqref{CCtime1}, and \eqref{CCtime2}, 
we have
\begin{align*}
\|u\|_{X^{s, b, \dl}} \sim \|v \|_{X^{s, b}}
\lesssim \dl^{\frac{1}{2}-b-} \|\wt{u} \|_{X^{s, \frac{1}{2}}}
\end{align*}

\noi
for any extension $\wt{u}$  such that $\wt{u} = u$ on $[-\dl, \dl]$.	
Therefore, \eqref{CCtime0} follows from the definition
\eqref{Xsb2}. 
\end{proof}

Lastly, we present the deterministic multilinear estimates.
We use them in High Modulation Case (Subsections \ref{SUBSEC:LWP3} 
and \ref{SUBSEC:GWP4}.)
Recall the periodic $L^4$-Strichartz estimate from \cite{Bourgain:1993p453}: 
\begin{equation}
	\label{L4} \|u\|_{L^4_{x, t}} \lesssim \|u\|_{X^{0, \frac{3}{8}}}. 
\end{equation}

\noi Interpolating \eqref{L4} with $\|u\|_{L^2_{x, t}} = \|u\|_{X^{0, 0}}$, we have 
\begin{equation}
	\label{L3} \|u\|_{L^{3+}_{x, t}} \lesssim \|u\|_{X^{0, \frac{1}{4}+}}, \text{ and } \, \|u\|_{L^{2+}_{x, t}} \lesssim \|u\|_{X^{0, 0+}}. 
\end{equation}

\begin{lemma} \label{LEM:deterministic}
Let $u_j$, $j = 1, 2, 3, 4$, be functions on $\T\times [-\dl, \dl]$.
Then, we have

\noi
\textup{(a)}
\begin{equation} \label{det1}
\int_{-\dl}^\dl \int_\T u_1 u_2 u_3 u_4 dx dt \lesssim \prod_{j = 1}^4 \| u_j\|_{X^{0, \frac{3}{8}, \dl}}.
\end{equation}

\noi
\textup{(b)} With large $p$, we have
\begin{equation} \label{det2}
\int_{-\dl}^\dl \int_\T u_1 u_2 u_3 u_4 dx dt \lesssim 
\prod_{j = 1}^3 \| u_j\|_{X^{0, \frac{1}{4}+, \dl}} \|u_4\|_{L^p(\T\times [-\dl, \dl])}.
\end{equation}

\noi
\textup{(c)} With large $p$, we have
\begin{equation} \label{det3}
\int_{-\dl}^\dl \int_\T u_1 u_2 u_3 u_4 dx dt \lesssim 
\prod_{j = 1}^2 \| u_j\|_{X^{0, 0+, \dl}}  \prod_{j = 3}^4 \|u_j\|_{L^p(\T\times [-\dl, \dl])}.
\end{equation}

\noi
\textup{(d)} With large $p$, we have
\begin{equation} \label{det4}
\int_{-\dl}^\dl \int_\T u_1 u_2 u_3 u_4 dx dt \lesssim 
\prod_{j = 1}^2 \| u_j\|_{X^{0, \frac{3}{8}, \dl}}  \| u_3\|_{X^{0, 0+, \dl}}  \|u_4\|_{L^p(\T\times [-\dl, \dl])}.
\end{equation}
\end{lemma}

\noi
Recall  that \eqref{det1} is the essential multilinear estimate for local well-posedness of the cubic NLS in $L^2(\T)$
by Bourgain \cite{Bourgain:1993p453}.

\begin{proof}
Let $\wt{u}_j$ be an extension of $u_j$ onto $\R$.
Then, by H\"older inequality and \eqref{L4}, we have
\begin{equation} 
\text{LHS of } \eqref{det1}
\leq \prod_{j = 1}^4 \|u_j\|_{L^4_{x, t}(\T\times [-\dl, \dl])}
\lesssim \prod_{j = 1}^4 \| \wt{u}_j\|_{X^{0, \frac{3}{8}, \dl}}.\label{det1-1}
\end{equation}

\noi
Hence, \eqref{det1} follows since \eqref{det1-1} holds for any extensions $\wt{u}_j$.
The other estimates \eqref{det2}, \eqref{det3}, and \eqref{det4}
follow in a similar manner by H\"older inequality:
\begin{align*}
1 & =  \tfrac{1}{3+}+\tfrac{1}{3+}+\tfrac{1}{3+}+\tfrac{1}{p} \quad \text{for }   \eqref{det2},\\
1 & =  \tfrac{1}{2+}+\tfrac{1}{2+}+\tfrac{1}{p}+\tfrac{1}{p} \quad \text{for } \eqref{det3},\\
1 & =  \tfrac{1}{4}+\tfrac{1}{4}+\tfrac{1}{2+}+\tfrac{1}{p} \quad \text{for }  \eqref{det4},
\end{align*}

\noi
with \eqref{L4} and \eqref{L3}. 
\end{proof}

\subsection{Probabilistic Lemmata}
In this subsection, we present several probabilistic lemmata related to the Gaussian random variables.
In the following, $\{g_n\}_{n\in\mathbb{Z}}$ denotes a family of independent standard complex valued Gaussian random variables
on a probability space $(\Omega, \mathcal{F}, \mathbb{P})$.
\begin{lemma}
\label{LEM:prob1} Let $\eps, \beta > 0$ and $\dl \ll 1$. Then, we have 
\begin{equation}
|g_n(\omega)| \leq C \dl^{-\frac{\beta}{2}} \jb{n}^\eps 
\end{equation}
	
\noi 
for all $n \in \mathbb{Z}$ for $\omega$ outside an exceptional set of measure $< e^{-\frac{1}{\dl^c}}$. 
\end{lemma}
\begin{proof}
	Recall from \cite{Oh:2009p791} that we have 
	$\mathbb{P}( \sup_n \jb{n}^{-\eps} |g_n(\omega)| > K ) \leq e^{-cK^2}$
	for sufficiently large $K>0$.	
	Now, choose $K \sim \dl^{-\frac{\beta}{2}}$. 
\end{proof}

\begin{lemma} \label{LEM:prob2} 
Let $f^\omega(x, t) = \sum c_n g_n(\omega) e^{i(nx + n^2t)} $. Then, for $p \geq 2$, there exists $\dl_0 > 0$ such that
	\[ \mathbb{P} ( \|f^\omega\|_{L^p(\mathbb{T}\times [-\dl, \dl])} > C\|c_n\|_{l^2_n} ) < e^{-\frac{1}{\dl^c} }\]
	
	\noi for $\dl \leq \dl_0$. 
\end{lemma}

\noi
This lemma is in the spirit of Paley-Zygmund \cite{PZ}.
In particular, it  says that the linear solution with random initial data
satisfies much better Strichartz estimates (with large probability.)
Compare with the deterministic case, where Strichartz estimates hold only for $p \leq 4$
(and for $p\leq 6$ with a slight loss of derivative.)

\begin{proof}
	By separating the real and imaginary parts, assume that $g_n$ is real-valued without loss of generality. From the general Gaussian bound (c.f. Burq-Tzvetkov \cite{Burq:2008p623}), there exists $C>0$ such that
	\[ \big\|\sum_n c_n g_n(\omega) \big\|_{L^r(\Omega)} \leq C \sqrt{r} \|c_n\|_{l^2_n}\]
	
	\noi for every $r \geq 2$ and every $\{c_n\}_{n \in \mathbb{Z}} \in l^2_n$. (This is also immediate from the hypercontractivity property as well. See \cite{Tzvetkov:2010p1443}.) By Minkowski integral inequality, we have 
	\begin{align*}
		\mathbb{E}\big( \|f^\omega\|_{L^p_{x, t}(\mathbb{T} \times [-\dl, \dl])}^r\big)^\frac{1}{r} &\leq \big\|\|f^\omega\|_{L^r(\Omega)}\big\|_{L^p_{x, t}} \lesssim \sqrt{r} \big\| \|c_n\|_{l^2_n} \big\|_{L^p_{x, t}(\mathbb{T} \times [-\dl, \dl])} \\
		&\lesssim \sqrt{r}\, \dl^\frac{1}{p} \|c_n\|_{l^2_n} 
	\end{align*}
	
	\noi for $r \geq p$. Then, by Chebyshev inequality, we have 
	\begin{align*}
		\mathbb{P} (\|f^\omega\|_{L^p(\mathbb{T} \times [-\dl, \dl])} > \ld ) \leq C^r \ld^{-r} r^\frac{r}{2} \dl^\frac{r}{p} \|c_n\|_{l^2_n}^r. 
	\end{align*}
	
	\noi Let $\ld = C r \dl^\frac{1}{p} \|c_n\|_{l^2_n}$ 
	and $ r = \dl^{c}$ with $c = \frac{1}{p}$. Then, we have
	\[ \mathbb{P} (\|f^\omega\|_{L^p(\mathbb{T} \times [-\dl, \dl])} > C  \|c_n\|_{l^2_n} ) 
	\leq e^{-r \ln \sqrt{r}} \leq e^{-\frac{1}{\dl^{c}}},\]
	
	\noi for $\dl$ sufficiently small such that $ r \geq p$. 
It follows from the proof that $\dl_0 \sim e^{-p \ln p}$.
\end{proof}

The following lemma 
follows from the hypercontractivity of the Ornstein-Uhlenbeck semigroup,
related to products of Gaussian random variables. 
See Ledoux-Talagrand \cite{Ledoux} and Janson \cite{Janson}.
A nice summary is given by Tzvetkov \cite[Sections 3 and 4]{Tzvetkov:2010p1443}.

\begin{lemma} \label{LEM:hyper}
For fixed $n \in \mathbb{Z}$, 
let 
\[ D_n = \{ (n_1, n_2, n_3) \in \mathbb{Z}^3: n = n_1 - n_2 + n_3, \  n_2 \ne n_1,\ n_2\ne n_3, \ n_1\ne n_3\}.\]

\noi
Given $\{a_{n_1, n_2, n_3}\} \in l^2(D_n)$, 
define $F_n$ by 
\begin{align*}
	F_n (\omega):= \sum_{\substack{n = n_1 - n_2 + n_3\\ n_2 \ne n_1, n_3\\n_1\ne n_3}} a_{n_1, n_2, n_3} g_{n_1}(\omega)\cj{g_{n_2}}(\omega)g_{n_3}(\omega).
\end{align*}

\noi
Then, 
there exists $c> 0$ such that, for $\ld > 0$, we have
\begin{equation}\label{eq:hyper}
 \mathbb{P} (|F_n(\omega)| \geq \ld) \leq \exp (-c \| F_n\|_{L^2(\Omega)}^{-\frac{2}{3}} \ld^{\frac{2}{3}}).
 \end{equation}

\end{lemma}

\begin{proof}
By Propositions 3.1 and 3.3 in \cite{Tzvetkov:2010p1443}, we have 
\[\|F_n\|_{L^p(\Omega)} \leq p^\frac{3}{2} \|F_n\|_{L^2(\Omega)},\] 

\noi
for all $2 \leq p <\infty$.
Then, \eqref{eq:hyper} follows from Lemma 4.5 in \cite{Tzvetkov:2010p1443}. 
\end{proof}

\section{Local Theory} \label{SEC:LWP}

\subsection{Basic Setup} \label{SUBSEC:LWP1}

Consider the Duhamel formulation \eqref{NLS3} of the Wick ordered NLS.
As mentioned before, when $\al \leq \frac{1}{2}$, 
the linear part $S(t)u_0^\omega \notin L^2(\T)$
almost surely.
Nonetheless,
we show that the nonlinear part lies in a smoother space $H^s (\mathbb{T})$ for some $s \geq 0$. 
More precisely, we prove that for each small $\dl> 0$, 
 there exists $\Omega_\dl$ with complemental measure $< e^{-\frac{1}{\dl^c}}$ 
such that $\G$ defined in \eqref{NLS3} is a contraction on $S(t) u_0^\omega + B$ for 
$\omega \in \Omega_\dl$, where $B$ denotes the ball of radius 1 in $X^{s, \frac{1}{2}+, \dl}$ for some $s \geq 0$.
i.e. we construct a contraction centered at the linear solution.

Given $u$ on $\T\times [-\dl, \dl]$, let $\wt{u}$ be an extension of $u$ onto $\T \times \R$.
By the nonhomogeneous linear estimate \cite{Bourgain:1993p453}, \cite{Ginibre:1997p1264},
we have
\begin{align}
	 	\bigg\| \int_0^t S(t - t') \mathcal{N}(u) (t') d t'\bigg\|_{X^{s, \frac{1}{2}+, \dl}}
	& \leq \bigg\|\eta_{_\dl}(t) \int_0^t S(t - t') \mathcal{N}(\wt{u}) (t') d t'\bigg\|_{X^{s, \frac{1}{2}+}} \notag \\
	& \lesssim \| \mathcal{N}(\wt{u}) \|_{X^{s, -\frac{1}{2}+}}, \label{duhamel}
\end{align}

\noi where $\eta_{_\dl}$ is a smooth cutoff on $[-2\dl, 2\dl]$. 
Then, Theorem \ref{THM:LWP} follows once we  prove 
\begin{equation}
	\label{trilinear1} \| \mathcal{N}(\wt{u}) \|_{X^{s, -\frac{1}{2}+}} \lesssim \dl^\theta, \quad \text{for some } \theta > 0 
\end{equation}

\noi for $\omega \in \Omega_\dl$ with $\mathbb{P}(\Omega^c_\dl) < e^{-\frac{1}{\dl^c}}$
(for {\it some} extension $\wt{u}$ of $u$.)
From the embedding $X^{s, \frac{1}{2}+, \dl} \subset C([-\dl, \dl]:H^s)$, 
it follows that \eqref{duhamel} and \eqref{trilinear1} imply that the nonlinear part of the solution $u^\omega$ is in 
 $C([-\dl, \dl]:H^s)$ with large probability.
Now, write $\mathcal{N}(u) $ as follows: 
\begin{align}
	\label{nonlinear1} \mathcal{N}(u) & = u |u|^2 - 2u \fint \ |u|^2 \notag \\
	& = \sum_{n_2 \ne n_1, n_3} \ft{u}(n_1)\cj{\ft{u}(n_2)}\ft{u}(n_3) e^{i(n_1 - n_2 + n_3)x} - \sum_n \ft{u}(n)|\ft{u}(n)|^2 e^{inx} =: \mathcal{N}_1(u) - \mathcal{N}_2(u). 
\end{align}

\noi 
Here, the condition $n_2 \ne n_1, n_3$ in the sum for $\mathcal{N}_1(u)$ 
is a shorthand notation for $n_2 \ne n_1$ and $n_2 \ne n_3$.
This shorthand notation is used in the remaining part of the paper.

In the following subsections, 
we will prove \eqref{trilinear1} by separately estimating the contributions from $\mathcal{N}_1(\wt{u}) $ and 
$\mathcal{N}_2(\wt{u}) $.
In particular, we choose an extension $\wt{u}$ in  $S(t) u_0^\omega +X^{s, \frac{1}{2}+}$
of 
$u \in S(t) u_0^\omega + B$. i.e. $u = S(t) u_0^\omega + v$ for some $v$ with 
$ \|v \|_{X^{s, \frac{1}{2}+, \dl}} \leq 1$.

By regarding $\mathcal{N}_1$ and $\mathcal{N}_2$
as trilinear operators, we write
\begin{align}
\label{NN1}
& \mathcal{N}_1(u_1, u_2, u_3) = \sum_{n_2 \ne n_1, n_3} \ft{u}_1(n_1, t)\cj{\ft{u}_2(n_2, t)}\ft{u}_3(n_3, t) e^{i(n_1-n_2+n_3)x}, \\
\label{NN2}
& \mathcal{N}_2(u_1, u_2, u_3) = \sum_n \ft{u}_1(n, t)\cj{\ft{u}_2(n, t)}\ft{u}_3(n, t) e^{inx}. 
\end{align}

\noi
Then, we prove \eqref{trilinear1}
by carrying out case-by-case analysis
on  \[\|\mathcal{N}_1(u_1, u_2, u_3)\|_{X^{s, -\frac{1}{2}+}}
\quad \text{ and }\quad
\|\mathcal{N}_2(u_1, u_2, u_3)\|_{X^{s, -\frac{1}{2}+}},\]

\noi
where $u_j$ is taken to be either of type
\begin{itemize}
\item[(I)] linear part: {\it random, less regular} 
	  \[\ u_j (x, t) = \sum_{n } \frac{g_n(\omega)}{\sqrt{1+|n|^{2\al}}} e^{i(nx + n^2t)}\]
\item[(II)]	 nonlinear part: {\it deterministic, smoother}
	 \[ \ u_j = \wt{v}_j \text{, where $\wt{v}_j$ is an extension of $v_j$ with } \|v_j \|_{X^{s, \frac{1}{2}+, \dl}} \leq 1. \]
\end{itemize}

\noi In the following, we may insert the smooth cutoff function $\eta_{_\dl}$ supported on $[-2\dl, 2\dl]$ 
(or the sharp cutoff function $\chi_{_\dl}$ supported on $[-\dl, \dl]$) without stating explicitly.
This merely corresponds to taking different extensions,
and does not cause any problem since our goal is to prove \eqref{trilinear1} for {\it some}
extension $\wt{u}$ of $u$.

Note that \eqref{duhamel} and \eqref{trilinear1} imply only the boundedness of the map $\G$ in \eqref{NLS3}
from $S(t) u_0^\omega + B$ into itself (for $\dl>0$ small). In establishing the contraction property, one needs to consider 
the difference $\G u_1 - \G u_2$ for $u_1, u_2 \in S(t) u_0^\omega + B$. We omit details since the computation  follows in a similar manner. Lastly, suppose that $u_0 = u_0^\omega$ is a good initial condition such that $\G$ is a contraction on $S(t) u_0 + B$. Let $\wt{u}_0$ be a function on $\mathbb{T}$ such that $\| u_0 - \wt{u}_0\|_{H^s} < \frac{1}{10} $. Denote by $\wt{\G}$ the solution map corresponding to the initial condition $\wt{u}_0$. Then, one can show that $\wt{\G}$ is also a contraction on $S(t) u_0 + B$ for $\dl$ sufficiently small. Moreover, we have
\[ \| u (t) - \wt{u}(t) \|_{H^s} \leq C \|u_0 - \wt{u}_0 \|_{H^s}\]

\noi for $|t| \leq \dl$, where $\wt{u}$ is the solution with the initial condition $\wt{u}_0$. For details, see \cite{Bourgain:1993p453}, \cite{Bourgain:1996p446}.

\subsection{Estimate on $\mathcal{N}_2$}  \label{SUBSEC:LWP2}

In this subsection, we prove the easier part of the estimate \eqref{trilinear1}: 
\begin{equation}
	\label{trilinear2} \| \mathcal{N}_2(u_1, u_2, u_3)\|_{X^{s, -\frac{1}{2}+}} \lesssim \dl^\theta 
\end{equation}

\noi for some $ \theta > 0$, 
outside an exceptional set of measure $e^{-\frac{1}{\dl^c}}$, where $\mathcal{N}_2$ is as in \eqref{NN2}
and $u_j$ is either of type (I) or (II).
We have 
\begin{equation}
	\text{LHS of } \eqref{trilinear2} = \bigg\| \frac{\jb{n}^s}{\jb{\tau - n^2}^{\frac{1}{2}-}} \intt_{\tau = \tau_1 - \tau_2 + \tau_3} \ft{u}_1(n, \tau_1)\cj{\ft{u}_2(n, \tau_2)}\ft{u}_3(n, \tau_3) d\tau_1 d\tau_2 \bigg\|_{l^2_n L^2_\tau} . 
\label{easy1} 
\end{equation}

\noi In the following,  we may replace $u_j$ by $\eta_{_\dl}u_j$ if necessary, 
where $\eta_{_\dl}$ the smooth cutoff function supported on $[-2\dl, 2\dl]$.

\medskip

\noi $\bullet$ {\bf Case (a):} $u_j$ of type (II), $j = 1, \dots, 3$.

By H\"older inequality with $p$ large ($\frac{1}{2} =\frac{1}{2+} + \frac{1}{p}$), we have 
\begin{align*}
	\eqref{easy1} \lesssim \sup_n \|\jb{\tau - n^2}^{-\frac{1}{2}+}\|_{L^{2+}_\tau} \Big\| \jb{n}^s\intt_{\tau = \tau_1 - \tau_2 + \tau_3} \ft{u}_1(n, \tau_1)\cj{\ft{u}_2(n, \tau_2)}\ft{u}_3(n, \tau_3) d\tau_1 d\tau_2 \Big\|_{l^2_{n} L^p_\tau}  
\end{align*}

\noi By Young and H\"older inequalities, 
\begin{align*}
	\lesssim \big\| \jb{n}^s \prod_{j = 1}^3 \| \ft{\wt{v}}_j (n, \tau) \|_{L^{\frac{3}{2}-}_\tau} \big\|_{l^2_n} \lesssim \big\| \jb{n}^s \prod_{j = 1}^3 \| \jb{\tau - n^2}^{\frac{1}{6}+} \ft{\wt{v}}_j (n, \tau) \|_{L^{2}_\tau} \big\|_{l^2_n} 
\end{align*}

\noi By  H\"older inequality and $l^2_n \subset l^6_n$, we have for $s \geq 0$ 
\begin{align*}
	\lesssim  \prod_{j = 1}^3 \| \jb{n}^\frac{s}{3} \jb{\tau - n^2}^{\frac{1}{6}+} \ft{\wt{v}}_j (n, \tau) \|_{l^6_n L^2_\tau} \leq  \prod_{j = 1}^3 \| \wt{v}_j \|_{X^{\frac{s}{3}, \frac{1}{6}+}},
\end{align*}

\noi
for any extension $\wt{v}_j$ of $v_j$ with  $\|v_j \|_{X^{s, \frac{1}{2}+, \dl}} \leq 1$.
Hence, by definition \eqref{Xsb2} and Lemma \ref{LEM:timedecay}, we have
\begin{align*}
	\eqref{easy1} \lesssim   
	\prod_{j = 1}^3 \|v_j \|_{X^{\frac{s}{3}, \frac{1}{6}+, \dl}}
	\lesssim \dl^{1-} \prod_{j = 1}^3 \| v_j \|_{X^{\frac{s}{3}, \frac{1}{2}+, \dl}} \leq \dl^{1-}. 
\end{align*}

\noi $\bullet$ {\bf Case (b):} $u_j$ of type (I), $j = 1, \dots, 3$.

By Lemma \ref{LEM:prob1}, we have $|g_n (\omega)| \leq C \dl^{-\frac{\beta}{2}} \jb{n}^\eps$ for $\omega$ outside an exceptional set of measure $< e^{-\frac{1}{\dl^c}}$. Then, by H\"older inequality with $p$ large ($\frac{1}{2} =\frac{1}{2+} + \frac{1}{p}$) and Young's inequality with \eqref{decay}, 
\begin{align*}
	\eqref{easy1} & \lesssim \sup_n \|\jb{\tau - n^2}^{-\frac{1}{2}+}\|_{L^{2+}_\tau}\\
	&\hphantom{XXX}\times \Big\| \jb{n}^{s-3\al} |g_n|^3 \intt_{\tau = \tau_1 - \tau_2 + \tau_3} \ft{\eta_{_\dl}}(\tau_1 - n^2)\cj{\ft{\eta_{_\dl}}(\tau_2 - n^2)}\ft{\eta_{_\dl}}(\tau_3 - n^2) d\tau_1 d\tau_2 \Big\|_{l^2_{n} L^p_\tau}\\
	& \lesssim \dl^{1-} \|\jb{n}^{s-3\al} |g_n (\omega)|^3 \|_{l^2_n} \lesssim \dl^{1 -\frac{3}{2}\beta-} \|\jb{n}^{s-3\al+ 3\eps} \|_{l^2_n} \lesssim \dl^{1-} 
\end{align*}

\noi as long as $2s - 6 \al + 6 \eps < -1$ or $ \al > \frac{1}{3}s + \frac{1}{6}$.

\noi $\bullet$ {\bf Case (c):} Exactly two $u_j$'s of type (I). Say $u_1(\I)$, $u_2(\I)$, and $u_3(\II)$. 

By H\"older inequality with $p$ large, Young's inequality with \eqref{decay}, and Lemma \ref{LEM:prob1}, we have 
\begin{align*}
	\eqref{easy1} & \lesssim \sup_n \|\jb{\tau - n^2}^{-\frac{1}{2}+}\|_{L^{2+}_\tau} \\
	& \hphantom{XX}\times \Big\| \jb{n}^{s-2\al} |g_n|^2 \intt_{\tau = \tau_1 - \tau_2 + \tau_3} \ft{\eta_{_\dl}}(\tau_1 - n^2)\cj{\ft{\eta_{_\dl}}(\tau_2 - n^2)}\ft{\wt{v}}_3(n, \tau_3) d\tau_1 d\tau_2 \Big\|_{l^2_{n} L^p_\tau}\\
	& \lesssim \dl^{\frac{1}{2}-} \big(\sup_n \jb{n}^{-2\al} |g_n|^2\big) \|\jb{n}^{s} \ft{\wt{v}}_3 (n, \tau) \|_{l^2_n L^2_\tau} 
	\lesssim \dl^{\frac{1}{2}-\beta-} \|\wt{v}_3 \|_{X^{s, 0}}, 
\end{align*}

\noi for $\al > 0$ outside an exceptional set of measure $< e^{-\frac{1}{\delta^c}}$,
where  $\wt{v}_3$ is any extension of $v_3$ with  $\|v_3 \|_{X^{s, \frac{1}{2}+, \dl}} \leq 1$.
Hence, by definition \eqref{Xsb2} and Lemma \ref{LEM:timedecay}, we have
\begin{align*}
\eqref{easy1} 
\lesssim \dl^{1 -\beta-} \|v_3 \|_{X^{s, \frac{1}{2}+, \dl}} \leq \dl^{1-} .
\end{align*}

\noi $\bullet$ {\bf Case (d):} Exactly one $u_j$ of type (I). Say $u_1(\I)$, $u_2(\II)$, and $u_3(\II)$.

By H\"older with $p$ large, Young's inequality with \eqref{decay}, and Lemma \ref{LEM:prob1}, we have 
\begin{align*}
	\eqref{easy1} & \lesssim \sup_n \|\jb{\tau - n^2}^{-\frac{1}{2}+}\|_{L^{2+}_\tau} \\
	& \hphantom{XX}\times \Big\| \jb{n}^{s-\al} |g_n| \intt_{\tau = \tau_1 - \tau_2 + \tau_3} \ft{\eta_{_\dl}}(\tau_1 - n^2)\cj{\ft{\wt{v}}_2(n, \tau_2)}\ft{\wt{v}}_3(n, \tau_3) d\tau_1 d\tau_2 \Big\|_{l^2_{n} L^p_\tau}\\
	& \lesssim \dl^{\frac{1}{2}-} \big(\sup_n \jb{n}^{-s-\al} |g_n|\big) \Big\| \prod_{j = 2}^3 \| \jb{n}^s \ft{\wt{v}}_j (n, \tau)\|_{L^\frac{4}{3}_\tau} \Big\|_{l^2_{n}}\\
	\intertext{By H\"older inequality in $n$ ($\frac{1}{2} = \frac{1}{4}+ \frac{1}{4}$) and in $\tau$ ($\frac{3}{4} = \frac{1}{2} + \frac{1}{4}$) with $l^2_n \subset l^4_n$,} 
	& \lesssim \dl^{\frac{1}{2}-\frac{\beta}{2} - } \prod_{j = 2}^3 \|\jb{n}^{s} \jb{\tau- n^2}^{\frac{1}{4}+} \ft{\wt{v}}_j (n, \tau) \|_{l^4_n L^2_\tau} 
	 \leq \dl^{\frac{1}{2}-\frac{\beta}{2} - } 
	\prod_{j = 2}^3 \|\wt{v}_j \|_{X^{s, \frac{1}{4}+}},
\end{align*}

\noi for $\al > -s$ outside an exceptional set of measure $< e^{-\frac{1}{\delta^c}}$,
where  $\wt{v}_j$, $j = 2, 3$ are any extensions of $v_j$ with  $\|v_j \|_{X^{s, \frac{1}{2}+, \dl}} \leq 1$.
Hence, by definition \eqref{Xsb2} and Lemma \ref{LEM:timedecay}, we have
\begin{align*}	
	\eqref{easy1} \lesssim \dl^{1-\frac{\beta}{2} - } 
	\prod_{j = 2}^3 \|v_j \|_{X^{s, \frac{1}{2}+, \dl}}
	 \lesssim \dl^{1-} .
\end{align*}

\begin{remark} \label{REM:local}\rm
In this subsection, we carefully used the definition \eqref{Xsb2} of the local-in-time space $X^{s, b, \dl}$.
Strictly speaking, such a care must be taken in all the subsequent nonlinear analysis.
However, this is a routine work and, for simplicity of presentation, we write estimates 
directly with 
$\|u_j\|_{X^{s, b, \dl}}$ in the following, 
meaning that the same estimates hold with $\|\wt{u}_j\|_{X^{s, b}}$ 
for any extension $\wt{u}_j$ of $u_j$
(and thus we can take the infimum over $\wt{u}_j$.)

\end{remark}

\subsection{Estimate on $\mathcal{N}_1 $: High Modulation Cases}  \label{SUBSEC:LWP3}

In the next two subsections, we prove the main part of the estimate \eqref{trilinear1}: 
\begin{equation}
	\label{trilinear3} \| \mathcal{N}_1(u_1, u_2, u_3) \|_{X^{s, -\frac{1}{2}+}} \lesssim \dl^\theta 
\end{equation}

\noi for some $ \theta > 0$, 
outside an exceptional set of measure $e^{-\frac{1}{\dl^c}}$, where $\mathcal{N}_1$ is as in \eqref{NN1}
and $u_j$ is either of type (I) or (II).

In (most of) the following,\footnote{Basically, we only need to dyadically decompose the function $u_j$ of type $(\I)$.} we assume that $u_1, u_2, u_3$
are dyadically decomposed with frequencies of size $N_1, N_2, N_3$, respectively.
As in \cite{Bourgain:1996p446}, 
let $N^1, N^2, N^3$ be the decreasing ordering of $N_1, N_2, N_3$ and $u^1, u^2, u^3$ be the corresponding $u_j$-factors. Also, let $\s^1, \s^2, \s^3$ denote the corresponding $\s_j := \jb{\tau_j - n_j^2}$. In the following, we use superscripts to imply that the functions (or variables) are arranged in the decreasing order of the spatial frequencies $N_1, N_2, N_3$.

In the rest of this subsection, we consider basic cases. 
Using duality, we can estimate \eqref{trilinear3} by 
\begin{equation}
	\label{duality1} \int_{-\dl}^\dl \int_\T (\jb{\dx}^s u^1) u^2 u^3 \cdot v \, dx dt 
\end{equation}

\noi where $\| v\|_{X^{0, \frac{1}{2}-, \dl}} \leq 1$ (with the complex conjugate on an appropriate $u^j$.)
Note that in \eqref{duality1}, we implicitly inserted the sharp cutoff function $\chi_{_\dl}$
(in one of the factors $u^j$.)

\medskip

\noi $\bullet$ {\bf Case (A):} $u^1$ and $u^2$ are of type $(\II)$.

Suppose that  $u^3$ is of type $(\II)$.
In this case, there is no need of apply dyadic on $u^1, u^2$, and $u^3$.
By Lemmata \ref{LEM:deterministic} (a) and  \ref{LEM:timedecay}, we have 
\begin{equation*}
	\eqref{duality1} 
	\lesssim \dl^{\frac{1}{2}-} \| u^1\|_{X^{s, \frac{1}{2}, \dl}} \| u^2\|_{X^{0, \frac{1}{2}, \dl}}\| u^3\|_{X^{0, \frac{1}{2}, \dl}} 
	\| v\|_{X^{0, \frac{1}{2}-, \dl}}
	\leq \dl^{\frac{1}{2}-} 
\end{equation*}

\noi as long as $s \geq 0$. 

Next, suppose that $u^3$ is of type $(\I)$ i.e. $u^3 = S(t) u_0$.
In this case, we do not need to apply dyadic decomposition on $u^1$ and $u^2$.
Namely, for a fixed dyadic block $N^3$ for $u^3$ of type $(\I)$, 
with a slight abuse of notation, 
we  use $u^1$ and $u^2$ to denote the sums of $u^j$ over the dyadic blocks $N^j \geq N^3$, $j = 1, 2$.

By Lemma \ref{LEM:deterministic} (b) with $p$ large followed by
Lemma \ref{LEM:prob2}, we have 
\begin{align*}
	\eqref{duality1} 
	& \lesssim  \| u^1\|_{X^{s, \frac{1}{4}+, \dl}} \| u^2\|_{X^{0, \frac{1}{4}+, \dl}} \| u^3\|_{L^p} 
	\| v\|_{X^{0,\frac{1}{4}+, \dl}} \\
	& \lesssim (N^3)^{\frac{1}{2}-\al+} \| u^1\|_{X^{s, \frac{1}{4}+, \dl}} \| u^2\|_{X^{0, \frac{1}{4}+, \dl}} \| v\|_{X^{0, \frac{1}{4}+, \dl}} 
\end{align*}

\noi outside an exceptional set of measure $< e^{-\frac{1}{\dl^c}}$. If $\jb{\tau_j - n_j^2}^{\frac{1}{4}-} \gtrsim (N^3)^{\frac{1}{2}-\al+}$ for $u_j$ of type $(\II)$, or if $\jb{\tau - n^2}^{\frac{1}{4}-} \gtrsim (N^3)^{\frac{1}{2}-\al+}$, then \eqref{trilinear3} follows with $\theta = \frac{1}{2}-$ in view of Lemma \ref{LEM:timedecay}.\footnote{This 
is to say that by inserting a cutoff (on the Fourier side)
on the region
satisfying $\jb{\tau_j - n_j^2}^{\frac{1}{4}-} \gtrsim (N^3)^{\frac{1}{2}-\al+}$ for $u_j$ of type $(\II)$, or $\jb{\tau - n^2}^{\frac{1}{4}-} \gtrsim (N^3)^{\frac{1}{2}-\al+}$, 
we can establish \eqref{trilinear3} with $\theta = \frac{1}{2}-$.}

Hence, it remains to estimate the contribution to \eqref{duality1}
with a cutoff (on the Fourier side)  
on the region
satisfying
\begin{equation}
	\label{case0} \jb{\tau - n^2} \ll (N^3)^{2-4\al+}, \text{ and } \jb{\tau_j - n_j^2} \ll (N^3)^{2-4\al+} \text{ if } u_j \text{ of type }(\II). 
\end{equation}

\medskip

\noi $\bullet$ {\bf Case (B):} $u^1$ of type $(\II)$, and $u^2$ of type $(\I)$.

In this case, we do not need to apply dyadic decomposition on $u^1$.
Namely, for a fixed dyadic block $N^2$ for $u^2$ of type $(\I)$, 
we use $u^1$ to denote the sum of $u^1$ over the dyadic blocks $N^1 \geq N^2$.

First, suppose that $u^3$ is of type $(\II)$. 
By Lemma \ref{LEM:deterministic} (b) with $p$ large followed by
Lemma \ref{LEM:prob2}, we have 
\begin{align*}
	\eqref{duality1} & \lesssim 
	\|u^1\|_{X^{s, \frac{1}{4}+, \dl}} \|u^2\|_{L^p} \|u^3\|_{X^{0, \frac{1}{4}+, \dl}}\|v\|_{X^{0, \frac{1}{4}+, \dl}} \\
	& \lesssim (N^2)^{\frac{1}{2}-\al+}\|u^1\|_{X^{s, \frac{1}{4}+, \dl}} \|u^3\|_{X^{0, \frac{1}{4}+, \dl}}
	\|v\|_{X^{0, \frac{1}{4}+, \dl}} 
\end{align*}

\noi outside an exceptional set of size $<e^{-\frac{1}{\dl^c}}$. If $\jb{\tau_j - n_j^2}^{\frac{1}{4}-} \gtrsim (N^2)^{\frac{1}{2}-\al+}$ for $u_j$ of type $(\II)$, or if $\jb{\tau - n^2}^{\frac{1}{4}-} \gtrsim (N^2)^{\frac{1}{2}-\al+}$, then \eqref{trilinear3} follows with $\theta = \frac{1}{2}-$ in view of Lemma \ref{LEM:timedecay}. 

Hence, it remains to estimate the contribution to \eqref{duality1}
from the region satisfying
\begin{equation}
	\label{caseA} \jb{\tau - n^2} \ll (N^2)^{2-4\al+}, \text{ and } \jb{\tau_j - n_j^2} \ll (N^2)^{2-4\al+} \text{ if } u_j \text{ of type }(\II) 
\end{equation}

\noi
in the following.

Next, suppose that $u^3$ is of type $(\I)$. 
By Lemma \ref{LEM:deterministic} (c) with $p$ large followed by
 Lemma \ref{LEM:prob2}, we have 
\begin{align*}
	\eqref{duality1} & \lesssim 
	\|u^1\|_{X^{s, 0+, \dl}} \|u^2\|_{L^p} \| u^3\|_{L^{p}}\|v\|_{X^{0, 0+, \dl}}\\
	& \lesssim (N^2)^{1-2\al+}\|u^1\|_{X^{s, 0+, \dl}}\|v\|_{X^{0, 0+, \dl}}. 
\end{align*}

\noi outside an exceptional set of measure $<e^{-\frac{1}{\dl^c}}$. If $(\s^1)^{\frac{1}{2}-} \gtrsim (N^2)^{1-2\al+}$ or if $\jb{\tau - n^2}^{\frac{1}{2}-} \gtrsim (N^2)^{1-2\al+}$, then \eqref{trilinear3} follows with $\theta = \frac{1}{2}-$ in view of Lemma \ref{LEM:timedecay}. 
Hence, it remains to estimate the contribution to \eqref{duality1}
from the region satisfying \eqref{caseA} as well.

\medskip

\noi $\bullet$ {\bf Case (C):} $u^1$ of type $(\I)$, and $u^2$, $u^3$ of type $(\II)$.

Suppose $\jb{\tau - n^2} \gg \max(\s^2, \s^3)$. 
By Lemma \ref{LEM:deterministic} (d) with $p$ large followed by
 Lemmata \ref{LEM:prob2} and \ref{LEM:timedecay}, we have 
\begin{align*}
	\eqref{duality1} & \leq (N^1)^s \|u^1\|_{L^p}
	\|u^2\|_{X^{0, \frac{3}{8}, \dl}}\|u^3\|_{X^{0, \frac{3}{8}, \dl}}\|v\|_{X^{0, 0+, \dl}} \\
	 & \lesssim (N^1)^{s+\frac{1}{2}-\al+}\|u^2\|_{X^{0, \frac{3}{8}, \dl}}\|u^3\|_{X^{0, \frac{3}{8}, \dl}}\|v\|_{X^{0, 0+, \dl}} \\
	&\lesssim \dl^{\frac{1}{4}-} (N^1)^{s+\frac{1}{2}-\al+}\|u^2\|_{X^{0, \frac{1}{2}, \dl}} \|u^3\|_{X^{0, \frac{1}{2}, \dl}}\|v\|_{X^{0, 0+, \dl}} 
\end{align*}

\noi outside an exceptional set of measure $<e^{-\frac{1}{\dl^c}}$. Hence, \eqref{trilinear3} follows as long as $\jb{\tau - n^2} \gtrsim (N^1)^{2s + 1 - 2\al+}$. Similar results hold if $\s^2 \gg \max(\s^3, \jb{\tau-n^2})$ 
or $\s^3 \gtrsim \max(\s^2, \jb{\tau-n^2})$. 

Hence, it remains to estimate the contribution to \eqref{duality1}
from the region satisfying
\begin{equation}
	\label{caseB} \jb{\tau - n^2} \ll (N^1)^{2s + 1 - 2\al+}, \text{ and } \jb{\tau_j - n_j^2} \ll (N^1)^{2s + 1 - 2\al+} \text{ if } u_j \text{ of type }(\II). 
\end{equation}

\medskip

\noi $\bullet$ {\bf Case (D):} $u^1$ of type $(\I)$, and either $u^2(\I)$, $u^3(\II)$ or $u^2(\II)$, $u^3(\I)$.

Suppose that $u^2$ is of type $(\I)$ and that $u^3$ is of type $(\II)$. Moreover, suppose $\jb{\tau - n^2} \gg \s^3$. 
By Lemma \ref{LEM:deterministic} (c) with $p$ large followed by
 Lemmata \ref{LEM:prob2} and \ref{LEM:timedecay}, we have 
\begin{align*}
	\eqref{duality1} & \leq (N^1)^s \|u^1\|_{L^p} \|u^2\|_{L^{p}}\|u^3\|_{X^{0, 0+, \dl}}\|v\|_{X^{0, 0, \dl}}\\
	& \lesssim (N^1)^{s+1-2\al+}\|u^3\|_{X^{0, 0+, \dl}}\|v\|_{X^{0, 0, \dl}} \\
	&\lesssim \dl^{\frac{1}{2}-} (N^1)^{s+1-2\al+}\|u^3\|_{X^{0, \frac{1}{2}, \dl}}\|v\|_{X^{0, 0, \dl}} 
\end{align*}

\noi outside an exceptional set of measure $<e^{-\frac{1}{\dl^c}}$. Hence, \eqref{trilinear3} follows as long as $\jb{\tau - n^2} \gtrsim (N^1)^{2s + 2 - 4\al+}$. Similar results hold if $\s^3 \gtrsim \jb{\tau-n^2}$, (or $u^2$ is of type $(\II)$ and $u^3$ is of type $(\I)$.) 

Hence, it remains to estimate the contribution to \eqref{duality1}
from the region satisfying
\begin{equation}
	\label{caseC} \jb{\tau - n^2} \ll (N^1)^{2s + 2 - 4\al+}, \text{ and } \jb{\tau_j - n_j^2} \ll (N^1)^{2s + 2 - 4\al+} \text{ if } u_j \text{ of type }(\II). 
\end{equation}

\medskip

\noi {\bf Summary:} Given a function $v(x, t)$, we can write $v$ as 
\begin{align}
	\label{v1} v (x, t)= \int \jb{\ld}^{-\frac{1}{2}-} \Big( \sum_n \jb{n}^{2s} \jb{\ld}^{1+} |\ft{v}(n, n^2 + \ld)|^2 \Big)^\frac{1}{2} \Big\{e^{i \ld t} \sum_n a_{\ld} (n) e^{i(nx + n^2 t)} \Big\} d \ld 
\end{align}

\noi where $a_{\ld} (n) = \frac{\ft{v}(n, n^2 + \ld)} {( \sum_m \jb{m}^{2s} |\ft{v}(m, m^2 + \ld)|^2 )^\frac{1}{2}}$. Note that $\sum_n \jb{n}^{2s} |a_{\ld}(n)|^2 = 1$. For $\|v\|_{X^{s, \frac{1}{2}+}}\leq 1$, we have 
\begin{equation}
	\label{v2} \int_{|\ld| < K } \jb{\ld}^{-\frac{1}{2}-} 
	\Big( \sum_n \jb{n}^{2s} \jb{\ld}^{1+} |\ft{v}(n, n^2 + \ld)|^2 \Big)^\frac{1}{2} d \ld \lesssim 1 
\end{equation}

\noi
by Cauchy-Schwarz inequality. See (22) and (23) in \cite{Bourgain:1996p446}. Note that \eqref{v1} is a standard representation for functions in $X^{s, b}$ for $b> \frac{1}{2}$. For example, see Klainerman-Selberg \cite{Klainerman:2002p743}. 

\medskip
\noi
$\bullet$ {\bf Case 1:}
First, we consider the case if any of $u_j$ is of type (II).
From Cases (A)--(D), we assume that 
$\jb{\tau - n^2} \ll K = K(N^j)$ for $j = 1, 2$, or $3$ in the following.
By H\"older inequality, we have
\begin{align*}
\|\N_1\|_{X^{s, -\frac{1}{2}+\eps}}
& = \bigg(\sum_n \int \jb{n}^{2s} \frac{|\ft{\N_1}(n, \tau)|^2}{\jb{\tau - n^2}^{1-2\eps}}d\tau \bigg)^\frac{1}{2}
= \bigg(\sum_n \int \jb{n}^{2s} \frac{|\ft{\N_1}(n, \ld + n^2)|^2}{\jb{\ld}^{1-2\eps}}d\ld \bigg)^\frac{1}{2}\\
& \lesssim K^\eps 
\big\|\jb{n}^s \ft{\N_1}(n, \ld + n^2) \big\|_{L^\infty_\ld l^2_n}.
\end{align*}

\noi Then, letting $* = \{ (n_1, n_2, n_3) \in\mathbb{Z}^3: n = n_1 - n_2 + n_3, \ n_2 \ne n_1, n_3\}$ 
and $**_n = \{ (\tau_1, \tau_2, \tau_3) \in \mathbb{R}^3: \tau = \ld + n^2 = \tau_1 - \tau_2 + \tau_3 \}$, we have 
\begin{align}
	 \text{LHS of } \eqref{trilinear3} 
	 & \lesssim K^{0+} \sup_{\jb{\ld}\ll K} \bigg\| \sum_{*} \jb{n^1}^s \int_{**_n}
	 \prod_{j = 1}^ 3 \ft{u}_j(n_j, \tau_j)d\tau_1 d \tau_2 \bigg\|_{l^2_n} ,
\label{reduction1}
\end{align}

\noi where $\ft{u}_j(n_j, \tau_j) = \frac{g_{n_j}(\omega)\dl (\tau_j - n_j^2)}{\sqrt{1 + |n_j|^{2\al}}}$ or
\[\ft{u}_j(n_j, \tau_j) = \int_{\{|\ld_j| < K \}} \jb{\ld_j}^{-\frac{1}{2}-} c_j(\ld_j) a_{\ld_j}(n_j) \dl (\tau_j - n_j^2 - \ld_j) d \ld_j \]

\noi
with $\sum_{n_j} \jb{n_j}^{2s} |a_{\ld_j}(n_j)|^2 \leq 1$ and  $c_j (\ld_j) = \Big(\sum_{m_j} \jb{m_j}^{2s} \jb{\ld_j}^{1+} |\ft{u}_j(m_j, n_j^2 + \ld_j)|^2 \Big)^\frac{1}{2}$ . 
For $j$ such that  $u_j$ is of type (II),  we can pull the integral in the corresponding $\ld_j$ outside the $l^2_n$-norm via Minkowski integral inequality. 
Then, for fixed $n$, $n_j$, $\ld$, and $\ld_j$, 
by integrating in $\tau_1$ and $\tau_2$, we obtain
\[ \int_{**_n} \prod_{j = 1}^3 \dl(\tau_j - n_j^2 - \wt{\ld}_j) d\tau_1d\tau_2
= \begin{cases}
1, & n^2 - n_1^2 + n_2^2 - n_3^2  + \ld - \wt{\ld}_1 + \wt{\ld}_2 - \wt{\ld}_3 = 0,\\
0, & \text{ otherwise,}
\end{cases}\]

\noi
where
$\wt{\ld}_j = 0$ or $\ld_j$, corresponding to type(I) or (II).

For example, consider the case when  $u_1$ and $u_2$ are of type (II) and $u_3$ is of type (I).
Then, from \eqref{reduction1}, we have\footnote{We  drop the complex conjugate in the following
when it plays no role.}  
\begin{align}
 \bigg\| \sum_{*} & \jb{n^1}^s \int_{**_n}
 	 \prod_{j = 1}^ 3 \ft{u}_j(n_j, \tau_j)d\tau_1 d \tau_2 \bigg\|_{l^2_n}\notag \\ 
		 & \lesssim  \int \prod_{j = 1}^2 \chi_{\{\jb{\ld_j} < K \}} \jb{\ld_j}^{-\frac{1}{2}-} c_j(\ld_j) \bigg\| \sum_{*} \jb{n^1}^s \prod_{k = 1}^ 2 a_{\ld_k}(n_k)\frac{g_{n_3}(\omega)}{\sqrt{1 + |n_3|^{2\al}}} \bigg\|_{l^2_n} d\ld_1d\ld_2 \notag \\
	& \lesssim  \sup_{\ld_1, \ld_2} \bigg\| \sum_{*} \jb{n^1}^s \prod_{j = 1}^2 a_{\ld_j}(n_j)\frac{g_{n_3}(\omega)}{\sqrt{1 + |n_3|^{2\al}}}\bigg\|_{l^2_n}, 
	\label{reduction2}
\end{align}

\noi where the last inequality follows from Cauchy-Schwarz inequality and \eqref{v2}. 
Note that if $u_j$ is supported on $[-\dl, \dl]$ in time,
then, in view of (the proof of) Lemma \ref{LEM:timedecay}, we can gain $\dl^\theta$ for small $\theta > 0$
in \eqref{reduction2}
by making the power in $\jb{\ld}^{-\frac{1}{2}-}$ 
slightly larger (keeping it less than $-\frac{1}{2}$)
in \eqref{v2}.

\medskip
\noi
$\bullet$ {\bf Case 2:}
Next, we consider the case when all $u_j$'s are of type (I).
From \eqref{ALGEBRA}, we have 
\begin{equation}
|\tau - n^2| = |2(n_2-n_1)(n_2-n_3)| \lesssim (N^1)^2.
\label{sigmabound}
\end{equation}

\noi
By an argument similar to the proof of Lemma \ref{LEM:timedecay}, 
we have
\begin{align*}
\|\N_1\|_{X^{s, -\frac{1}{2}+\eps,\dl}}
& \lesssim \dl^\theta \|\N_1\|_{X^{s, -\frac{1}{2}+2\eps,\dl}}\\
& \leq \dl^\theta 
\bigg\|\jb{\tau- n^2}^{-\frac{1}{2}+2\eps} \sum_* \jb{n^1}^s 
\int_{**_n} \prod_{j = 1}^3 \frac{g_{n_j}\dl(\tau_j - n_j^2)}{\sqrt{1+|n_j|^{2\al}}}
d\tau_1d\tau_2\bigg\|_{l^2_n L^2_\tau}
\end{align*}

\noi
for small $\theta > 0$.
By integrating in $\tau_1$ and $\tau_2$, we have
\[ \int_{**_n} \prod_{j = 1}^3 \dl(\tau_j - n_j^2) d\tau_1d\tau_2
= \dl(\tau - n_1^2 + n_2^2 - n_3^2).\]

\noi
Hence, for fixed $n$, we have
\begin{align*}
 \bigg\|\jb{\tau- & n^2}^{-\frac{1}{2}+2\eps} \sum_* \jb{n^1}^s 
\int_{**_n} \prod_{j = 1}^3 \frac{g_{n_j}\dl(\tau_j - n_j^2)}{\sqrt{1+|n_j|^{2\al}}}
d\tau_1d\tau_2\bigg\|_{ L^2_\tau}\\
& = 
\bigg( \int \jb{\tau- n^2}^{-1+4\eps} 
\bigg|\sum_{\substack{n = n_1 - n_2 + n_3\\  n_2 \ne n_1, n_3}}
\jb{n^1}^s 
 \prod_{j = 1}^3 \frac{g_{n_j}}{\sqrt{1+|n_j|^{2\al}}}
\dl(\tau - n_1^2 + n_2^2 - n_3^2)\bigg|^2d\tau\bigg)^\frac{1}{2}\\
& = 
\bigg( \int_{|\mu| \lesssim (N^1)^2}
\jb{\mu}^{-1+4\eps} \bigg)^\frac{1}{2}
\sup_{|\mu| \lesssim (N^1)^2}
\bigg|\sum_{\substack{n = n_1 - n_2 + n_3\\  n_2 \ne n_1, n_3\\n^2 = n_1^2 + n_2^2 - n_3^2+ \mu}}
\jb{n^1}^s 
 \prod_{j = 1}^3 \frac{g_{n_j}}{\sqrt{1+|n_j|^{2\al}}}
\bigg|\\
& \lesssim
(N^1)^{s+4\eps}
\sup_{|\mu| \lesssim (N^1)^2}
\bigg|\sum_{\substack{n = n_1 - n_2 + n_3\\  n_2 \ne n_1, n_3\\n^2 = n_1^2 + n_2^2 - n_3^2+ \mu}}
 \prod_{j = 1}^3 \frac{g_{n_j}}{\sqrt{1+|n_j|^{2\al}}}
\bigg|.
\end{align*}

\noi
Then, we can take $l^2$-summation in $n$.

\medskip

Therefore, we can reduce the estimate \eqref{trilinear3} into the following two cases (with $\theta = 0+$):

\medskip

\noi $\bullet$ $u^1$ is of type $(\II)$:

In this case, we do not need to apply dyadic decomposition on $u^1$.
Namely, for a fixed dyadic block $N^2$ for $u^2$, 
with a slight abuse of notation, 
we use $u^1$ to denote the sum of $u^1$ over the dyadic blocks $N^1 \geq N^2$.

From \eqref{case0} and \eqref{caseA}, we can assume that $\s_j \ll (N^3)^{2-4\al+}$ or $(N^2)^{2-4\al+}$ for $u_j$ of type $(\II)$. Then, by \eqref{v1} and \eqref{v2}, we can bound \eqref{trilinear3} as follows: 
\begin{align}
	 \eqref{trilinear3} \lesssim \dl^{\theta} M(N^2, N^3) \bigg( \sum_n \Big| \sum_{\substack{ n = n_1 - n_2 + n_3 \\
	n_2 \ne n_1, n_3\\
	n^2 = n_1^2 - n_2^2 + n_3^2 + \mu}} a_1(n_1)\cj{a_2(n_2)}a_3(n_3) \Big|^2 \bigg)^\frac{1}{2}, 
\label{u2}
\end{align}

\noi where $\sum_{n} |a^1(n)|^2 \leq 1$, $a^j(n) = \frac{g_n(\omega)}{\sqrt{1 + |n|^{2\al}}}$ or $\sum_{|n| \sim N^j} |a^j(n)|^2 \leq (N^j)^{-2s}$ for $j = 2, 3$, and 
\begin{tabbing}
	\hspace{1cm} \= Case (A): \= $M(N^2, N^3) = ( N^3)^{0+}$ \= and \= $|\mu| \ll (N^3)^{2-4\al+}$ \\
	\> Case (B): \> $M(N^2, N^3) = ( N^2)^{0+}$ \> and \> $|\mu| \ll (N^2)^{2-4\al+}$. 
\end{tabbing}


\medskip 

\noi $\bullet$ $u^1$ is of type $(\I)$:

From \eqref{caseB} and \eqref{caseC}, we can assume that $\s_j \ll (N^1)^{2s + 1 - 2 \al+}$ or $(N^1)^{2s + 2 - 4 \al+}$ for $u_j$ of type $(\II)$. Then, by \eqref{v1} and \eqref{v2}, we can bound \eqref{trilinear3} as follows: 
\begin{align}
	 \eqref{trilinear3} \lesssim \dl^{\theta} (N^1)^{s+} \bigg( \sum_{|n| \lesssim N^1} \Big| \sum_{\substack{ n = n_1 - n_2 + n_3 \\
	n_2 \ne n_1, n_3\\
	n^2 = n_1^2 - n_2^2 + n_3^2 + \mu}} a_1(n_1)\cj{a_2(n_2)}a_3(n_3) \Big|^2 \bigg)^\frac{1}{2}, 
\label{u1}
\end{align}

\noi where $a^1(n) = \frac{g_n(\omega)}{\sqrt{1 + |n|^{2\al}}}$, 
$a^j(n) = \frac{g_n(\omega)}{\sqrt{1 + |n|^{2\al}}}$ or $\sum_{|n| \sim N^j} |a^j(n)|^2 \leq (N^j)^{-2s}$ for $j = 2, 3$, and 
\begin{tabbing}
	\hspace{1cm} \= Case (C): \hspace{5mm} \= $|\mu| \ll (N^1)^{2s + 1 - 2 \al+}$ \\
	\> Case (D): \> $|\mu| \ll (N^1)^{2s + 2 - 4 \al+}$ \\
	\> All type (I): \> $|\mu| \ll (N^1)^2$, 
\end{tabbing}

\noi Note that all the spatial frequencies are dyadically decomposed in this case.

\medskip

Suppose $|n_2| > 10 (|n_1| + |n_3|)$. Then, on the one hand, $|\mu| \sim |(n_2 - n_1) (n_2 - n_3)| \sim |n_2|^2 \sim (N^1)^2$ by \eqref{ALGEBRA}. 
On the other hand, if $u^1 = u_2$ is of type $(\II)$, we have $|\mu| \ll (N^2)^{2- 4\al+} \ll (N^1)^2$ as long as $\al > 0$. If $u^1 = u_2$ is of type $(\I)$, we have $|\mu| \ll (N^1)^{2s + 2 - 4\al+} \ll (N^1)^2$ since $\al > \frac{s}{2}$. In both cases, we would have a contradiction. Hence, we can assume that $|n_1| \sim N^1$ or $|n_3| \sim N^1$. Moreover, by symmetry between $u_1$ and $u_3$, we assume $|n_1| \sim N^1$ in the following.

Lastly, we list all the different cases following \cite{Bourgain:1996p446}. We consider these cases in details in the next subsection.
\begin{tabbing}
	\hspace{1cm} \=Case (a): \= $n_1 = N^1(\II)$, \= $n_2 = N^2(\I)$, \= $n_3 = N^3(\II)$ or \= $n_2 = N^3(\I)$, \= $n_3 = N^2(\II)$ \\
	
	\>Case (b): \>$n_1 = N^1(\II)$, \>$n_2 = N^3(\II)$, \> $n_3 = N^2(\I)$ or \>$n_2 = N^2(\II)$, \> $n_3 = N^3(\I)$ \\
	
	\>Case (c): \> $n_1 = N^1(\I)$, \>$n_2 = N^2(\II)$, \>$n_3 = N^3(\II)$ \\
	
	\>Case (d): \>$n_1 = N^1(\I)$, \>$n_2 = N^3(\II)$, \>$n_3 = N^2(\II)$\\
	
	\>Case (e): \>$n_1 = N^1(\II)$, \>$n_2 = N^2(\I)$, \>$n_3 = N^3(\I)$\\
	
	\>Case (f): \>$n_1 = N^1(\II)$, \>$n_2 = N^3(\I)$, \>$n_3 = N^2(\I)$\\
	
	\>Case (g): \>$n_1 = N^1(\I)$, \>$n_2 = N^2(\I)$, \>$n_3 = N^3(\II)$\\
	
	\>Case (h): \>$n_1 = N^1(\I)$, \>$n_2 = N^3(\I)$, \>$n_3 = N^2(\II)$\\
	
	\>Case (i): \>$n_1 = N^1(\I)$, \>$n_2 = N^2(\II)$, \>$n_3 = N^3(\I)$\\
	
	\>Case (j): \>$n_1 = N^1(\I)$, \>$n_2 = N^3(\II)$, \>$n_3 = N^2(\I)$\\
	
	\>Case (k): \>$n_1 = N^1(\I)$, \>$n_2 = N^2(\I)$, \>$n_3 = N^3(\I)$\\
	
	\>Case (l): \>$n_1 = N^1(\I)$, \>$n_2 = N^3(\I)$, \>$n_3 = N^2(\I)$\\
\end{tabbing}

\subsection{Estimate on $\mathcal{N}_1$: Low Modulation Cases}  \label{SUBSEC:LWP4}

For notational simplicity, we use $|n|^\al $ for $\sqrt{1+ |n|^{2\al}}$. 
We may drop a complex conjugate on $u_2$ when it plays no significant role.
Now, let 
\begin{align*} 
	A_n = \{ (n_1, n_2, n_3) \in \mathbb{Z}^3: n = & \ n_1 - n_2 + n_3, |n_j| \sim N_j , j = 2, 3, \\
	& n_2 \ne n_1, n_3, \text{ and } n^2 = n_1^2 - n_2^2 + n_3^2 + \mu\} 
\end{align*}

\noi and $ B_n = A_n \cap \{ |n_1| \sim N_1 \}. $ Also, from \eqref{ALGEBRA}, we have 
\begin{equation}
	\label{ALGEBRA7} \mu = 2 (n_2 - n_1) (n_2 - n_3) = 2 (n - n_1) (n - n_3). 
\end{equation}

\medskip

\noi $\bullet$ {\bf Cases (k), (l):} $u_1, u_2, u_3$ of type $(\I)$. 
\quad
In this case, we have 
\begin{align}
	\label{casek} \eqref{u1} \lesssim \dl^{\theta} N_1^{s+} \bigg( \sum_{|n| \lesssim N_1} \Big|\sum_{B_n} \frac{g_{n_1}}{|n_1|^\al} \frac{\cj{g_{n_2}}}{|n_2|^\al} \frac{g_{n_3}}{|n_3|^\al} \Big|^2 \bigg)^\frac{1}{2}. 
\end{align}

First, we consider the contribution from $n_1 \ne n_3$. Let 
\begin{align*}
	F_n (\omega):= \sum_{C_n} \frac{g_{n_1}(\omega)}{|n_1|^\al} \frac{\cj{g_{n_2}}(\omega)}{|n_2|^\al} \frac{g_{n_3}(\omega)}{|n_3|^\al}, 
\end{align*}

\noi where $ C_n = B_n \cap \{ n_1 \ne n_3\}. $ Then, 
applying Lemma \ref{LEM:hyper} for
with  $\ld = \dl^{-\frac{3}{2}\beta} N_1^{\frac{3}{2}\eps} \| F_n\|_{L^2(\Omega)}$ with $\eps = 0+$, we have 
\begin{align}
	\label{chaosestimate} \mathbb{P} (|F_n(\omega)| \geq \dl^{-\frac{3}{2}\beta} N_1^{\frac{3}{2}\eps}\| F_n\|_{L^2(\Omega)}) \leq e^{-\frac{c'N_1^{\eps}}{ \dl^\beta}}. 
\end{align}

\noi By Lemma \ref{LEM:count1}, we have 
\begin{align*}
	\text{RHS of } \eqref{casek} & \lesssim \dl^{\theta-\frac{3}{2}\beta} N_1^{s+\frac{3}{2}\eps+} \bigg( \sum_{|n| \lesssim N_1} \sum_{C_n} \frac{1}{|n_1|^{2\al}|n_2|^{2\al}|n_3|^{2\al}} \bigg)^\frac{1}{2}\\
	& \lesssim \dl^{\theta-\frac{3}{2}\beta} N_1^{s -\al +\frac{3}{2}\eps+} (N^2)^{-\al} (N^3)^{-\al+\frac{1}{2}} \\
	& \lesssim 
	\begin{cases}
		\dl^{\theta-\frac{3}{2}\beta} N_1^{s -3\al + \frac{1}{2} +\frac{3}{2}\eps+} & \text{ for } \al \leq \frac{1}{4} \\
		\dl^{\theta-\frac{3}{2}\beta} N_1^{s -\al +\frac{3}{2}\eps+} & \text{ for }\al \geq \frac{1}{4} 
	\end{cases}
	\\
	& \leq \dl^{\theta-\frac{3}{2}\beta} \prod_{j = 1}^3 N_j^{0-}, \ \left\{
	\begin{array}{ll}
		\text{for } \al > \frac{s}{3} + \frac{1}{6} & \text{ (with } \al \leq \frac{1}{4}) \\
		\text{for } \al > s & \text{ (with }\al \geq \frac{1}{4}) 
	\end{array}
	\right. \notag 
\end{align*}

\noi outside an exceptional set of measure
\begin{equation}
\label{EXCEPT1}
< \sum_{|n| \lesssim N_1} e^{-\frac{c'N_1^{\eps}}{\dl^\beta}} \lesssim N_1 e^{-\frac{c'N_1^{\eps}}{\dl^\beta}} \leq N_1^{0-} e^{-\frac{c'}{\dl^\beta} N_1^{\eps} +(1+) \log N_1} < N_1^{0-} e^{-\frac{1}{\dl^{c}}}.
\end{equation}

\noi Note that in this case we need to make sure that the measures of these exceptional sets corresponding to different dyadic blocks are indeed summable and bounded by $e^{-\frac{1}{\dl^c}}$. We may not be explicit about this point in other cases. e.g. Cases (A)--(D) in Subsection \ref{SUBSEC:LWP3}. We do not encounter this issue in using Lemma \ref{LEM:prob1} since it gives one exceptional set of measure $<e^{-\frac{1}{\dl^c}}$ for all the frequencies.

Now, consider the contribution from $n_1 = n_3$. It follows from \eqref{ALGEBRA7} that there is at most one choice of $(n_1, n_2, n_3)$ for each fixed $n$. Thus, $\sum_{|n| \lesssim N_1} \big|\sum_{B_n, \, n_1 = n_3} 1\big|^2 = \sum_{|n| \lesssim N_1} \sum_{B_n, \, n_1 = n_3} 1$. Hence, by Lemmata \ref{LEM:count1} and \ref{LEM:prob1}, we have 
\begin{align}
	\text{RHS of } \eqref{casek} & \lesssim \dl^{\theta-\frac{3}{2}\beta} N_1^{s -2\al + 2\eps+} N_2^{-\al + \frac{1}{2} + \eps} \leq \dl^{\theta-\frac{3}{2}\beta}\prod_{j = 1}^3 N_j^{0-} 
	\label{ZK}
\end{align}

\noi  for $\al > \frac{s}{3} + \frac{1}{6}$ outside an exceptional set of measure $< e^{-\frac{1}{\dl^c}}$.

\medskip

\noi $\bullet$ {\bf Case (a):} (Case (b) can be treated in a similar way by replacing $n_2$ and $n_3$.)

In this case, we have $\mu = 2 (n_2 - n_1) (n_2 - n_3) = o( (N_2)^{2-4\al+})$. 
This implies that $|n|, |n_1|, |n_3| \lesssim N_2^q$ for some $q > 0$ since $n_2 \ne n_1, n_3$. 
Now, fix $n$.
Then, it follows from \eqref{DIVISOR}
that 
\begin{equation}\label{A1}
\sum_{A_n} 1 \lesssim N_2^\eps.
\end{equation}

\noi
Then, by Lemma \ref{LEM:prob1} and Cauchy-Schwarz inequality, we have 
\begin{align}
	\label{caseaa} \eqref{u2} & 
	\lesssim \dl^{\theta-\frac{\beta}{2}} (N_2)^{-\al + \frac{1}{2}\eps +} 
	\bigg( \sum_n  \Big(\sum_{A_n}|a_1(n_1)|^2|a_3(n_3)|^2\Big)
	 \Big(\sum_{A_n}1\Big)\bigg)^\frac{1}{2} \notag \\
	\intertext{By  \eqref{A1}, we have } 
	& \lesssim \dl^{\theta-\frac{\beta}{2}} N_2^{-\al+\eps+} 
	\Big( \sum_n  \sum_{A_n}|a_1(n_1)|^2|a_3(n_3)|^2
	 \Big)^\frac{1}{2} \notag \\
		& \lesssim \dl^{\theta-\frac{\beta}{2}} N_2^{-\al+\eps+} N_3^{-s}\leq \dl^{\theta'} N_2^{0-} N_3^{0-} 
\end{align}

\noi for $\al > 0$ and $s\geq 0 $ outside an exceptional set of measure $< e^{-\frac{1}{\dl^c}}$.
Note that $|n_3| \lesssim N_2^q$ is crucial in the last inequality of \eqref{caseaa} when $s = 0$, $n_2 = N^3$, and $n_3 = N^2$. 

In Case (b), we have $\mu = 2 (n_2 - n_1) (n_2 - n_3) = o( (N_3)^{2-4\al+})$, which implies that $|n|, |n_1|, |n_2| \lesssim N_3^q$ for some $q > 0$ since $n_2 \ne n_1, n_3$. The rest of the argument follows as above by replacing $n_2$ and $n_3$.

\medskip

\noi $\bullet$ {\bf Case (c):} (Case (d) can be treated in a similar way by replacing $n_2$ and $n_3$.)

Let $b_2(n_2) = |n_2|^s a_2(n_2)$. Then, we have $\sum_{|n_2|\sim N_2} |b_2(n_2)|^2 \lesssim 1$. By Lemma \ref{LEM:prob1} and Cauchy-Schwarz inequality on $n_3$ in the inner sum, 
\begin{align*}
	\eqref{u1} \lesssim \dl^{\theta-\frac{\beta}{2}} N_1^{s -\al +\eps+} N_2^{-s} N_3^{-s} \Big( \sum_{|n| \lesssim N^1} \sum_{B_n} |b_2(n_2)|^2 \Big)^\frac{1}{2} 
\end{align*}

\noi outside an exceptional set of measure $< e^{-\frac{1}{\dl^c}}$. 
For fixed $n_2$, it follows from (the proof of) Lemma \ref{LEM:count1} that there are at most $N_1^{0+}$ terms in the sum. 
Hence, we have 
\begin{align*}
	\eqref{u1} \lesssim \dl^{\theta-\frac{\beta}{2}} N_1^{s -\al +\eps+} 
	N_2^{-s} N_3^{-s} \leq \dl^{\theta'}N_1^{0-} N_2^{0-} N_3^{0-} 
\end{align*}

\noi for $\al > s \geq 0$.

\medskip

\noi $\bullet$ {\bf Case (e):} (Case (f) is basically the same.)

In this case, we have $ |\mu| = |2 (n_2 - n_1) (n_2 - n_3)| \ll N_2^{2-4\al+} $. 
Fix $n$.
Then, from \eqref{DIVISOR}, 
there are at most $d(\mu)  = O(N_2^{0+})$ many choices for $n_2$ and $n_3$.
Then, by Lemma \ref{LEM:prob1}, Cauchy-Schwarz inequality, and \eqref{A1}, we have
\begin{align*}
	\eqref{u2} 
& \lesssim \dl^{\theta-\beta} 
	N_2^{-\al+\eps+} N_3^{-\al} \bigg( \sum_{|n| \lesssim N_2^q} 
	\Big(\sum_{A_n} |a_1(n_1)|^2\Big) \Big(\sum_{A_n} 1\Big) \bigg)^\frac{1}{2} \\
& \lesssim \dl^{\theta-\beta} 
	N_2^{-\al+\frac{3}{2}\eps+} N_3^{-\al} \Big( \sum_{|n| \lesssim N_2^q} \sum_{A_n} |a_1(n_1)|^2 \Big)^\frac{1}{2} \\
& \lesssim \dl^{\theta-\beta} N_2^{-\al+2\eps+} N_3^{ -\al} \leq \dl^{\theta'}N_2^{0-} N_3^{0-} 
\end{align*}

\noi for $ \al > 0$ outside an exceptional set of measure $<e^{-\frac{1}{\dl^c}}$
(with some $q>0$ as in Case (a).)
In the above computation, we used 
\[ \sum_{|n| \lesssim N_2^q} \sum_{A_n} |a_1(n_1)|^2
\lesssim N_2^\eps \sum_{n_1} |a_1(n_1)|^2 \leq N_2^\eps\]

\noi
by first summing over $n_2$ (for fixed $n_1$) and then over $n_1$.

\medskip

\noi $\bullet$ {\bf Case (g):} (Cases (h), (i), (j)  are basically the same.)

Fix $n$.
Then, from \eqref{DIVISOR}, 
there are at most $d(\mu)  = O(N_1^{0+})$ many choices for $n_2$ and $n_3$.
Thus, we have $\sum_{A_n} 1 \lesssim N_1^\eps.$
Then, by Lemma \ref{LEM:prob1} and Cauchy-Schwarz inequality as before,  we have
\begin{align*}
	\eqref{u1} 
& \lesssim \dl^{\theta-\beta} 
	N_1^{s-\al+\eps+} N_2^{-\al+ \frac{1}{2}\eps} \Big( \sum_{|n| \lesssim N_1} 
	\sum_{A_n} |a_3(n_3)|^2\Big)^\frac{1}{2} \\
& \lesssim \dl^{\theta-\beta} 
	N_1^{s-\al+\frac{3}{2}\eps+} N_2^{-\al+ \frac{1}{2}\eps} N_3^{-s}
 \leq \dl^{\theta'}\prod_{j = 1}^3 N_j^{0-} 
\end{align*}

\noi for $ \al > s\geq 0$ outside an exceptional set of measure $<e^{-\frac{1}{\dl^c}}$.

\medskip

This completes the proof of Theorem \ref{THM:LWP}.

\section{Global Theory}  \label{SEC:GWP}

In this section, we prove almost sure global well-posedness of \eqref{NLS2}.

\subsection{Reduction of Theorem \ref{THM:GWP1} to Proposition \ref{PROP:HNLS} }  \label{SUBSEC:GWP1}

In this subsection, we first prove Theorem \ref{THM:GWP1}, assuming the following proposition.
Heuristically speaking, 
this  says 
``almost'' almost sure global well-posedness (Proposition \ref{THM:GWP2})
implies almost sure global well-posedness (Theorem \ref{THM:GWP1}.)

\begin{proposition}
	\label{THM:GWP2} Let $\al \in ( \frac{5}{12}, \frac{1}{2}]$. Given $T> 0$ and $\eps >0$, there exists $\Omega_{T, \eps} \in \mathcal{F}$ with the following properties:
	\begin{enumerate}
		\item[(i)] $\mathbb{P}(\Omega_{T, \eps}^c) = \rho_\al \circ u_0(\Omega_{T, \eps}^c) < \eps$, 
				where $\rho_\al$ is the Gaussian probability measure on $H^{\al-\frac{1}{2}-}(\T)$ defined in \eqref{Gaussian1}
		and $u_0$ is viewed as a map $u_0:\Omega \to H^{\al-\frac{1}{2}-}(\mathbb{T})$.
%
		
		\item[(ii)] For each $\omega \in \Omega_{T, \eps}$ there exists a (unique) solution $u$ of \eqref{NLS2} in
		\[e^{-i \dx^2 t}u_0 + C([-T, T];L^2(\mathbb{T})) \subset C([-T, T];H^{\al - \frac{1}{2}-}(\mathbb{T}))\]
		with the initial condition $u_0^\omega$ given by \eqref{IV}. 
			Here, the uniqueness holds in a very mild sense. See Remark \ref{REM:unique}.
	\end{enumerate}
\end{proposition}
\begin{proof}
	[Proof of Theorem \ref{THM:GWP1}] For fixed $\eps > 0$, let $T_j = 2^j$ and $\eps_j = 2^{-j} \eps$. Apply Proposition \ref{THM:GWP2} and construct $\Omega_{T_j, \eps_j}$. Then, let $\Omega_\eps = \bigcap_{j = 1}^\infty \Omega_{T_j, \eps_j}$. Note that \eqref{NLS2} is globally well-posed on $\Omega_\eps$ with $\mathbb{P} (\Omega_\eps^c) < \eps$. Now, let $\wt{\Omega} = \bigcup_{\eps > 0} \Omega_\eps$. Then, \eqref{NLS2} is globally well-posed on $\wt{\Omega}$ and $\mathbb{P} (\wt{\Omega}^c) = 0$. 
\end{proof}

Now, we present the proof of Proposition \ref{THM:GWP2}. 
\begin{proof}
	[Proof of Proposition \ref{THM:GWP2}]
	
	First, recall the following argument which relates the time of local existence $\dl$ and the size of the initial condition. Consider \eqref{NLS2}. 
We briefly review the deterministic $L^2$-local theory. 
For $t \in [-\dl, \dl]$, \eqref{NLS3} is equivalent to 	
\[ u(t) = S(t) u_0 \pm i \int_0^t S(t - t') \mathcal{N}(\chi_\dl(t') u) (t') d t'.\] 

\noi
By \eqref{duhamel},  we have 
	\begin{align*}
		 \| u \|_{X^{0, \frac{1}{2}+\eps_1, \dl}} 
		& \leq C_0\|u_0\|_{L^2} + C_1\|  \mathcal{N}(\chi_{_\dl} \wt{u})\|_{X^{0, -\frac{1}{2}+\eps_1}} \\
\intertext{for any extension $\wt{u}$ of $u$.
By duality (against $v$ in $X^{0, \frac{1}{2}-\eps_1}$)
with Lemma \ref{LEM:deterministic} (a) followed by Lemma \ref{LEM:timedecay},}		
		& \leq C_0\|u_0\|_{L^2} + C_2\dl^{\frac{1}{2}-\eps_1-\eps_2}\|\wt{u}\|^3_{X^{0, \frac{1}{2}+\eps_1}} 
	\end{align*}
	
	\noi 
	for some small $\eps_2>0$. Hence, we obtain
		\begin{align}
	\label{local} 	 \| u \|_{X^{0, \frac{1}{2}+, \dl}} 
		& \leq C_0\|u_0\|_{L^2} + C_2\dl^{\frac{1}{2}-\eps_1-\eps_2}\|u\|^3_{X^{0, \frac{1}{2}+\eps_1, \dl}}.
	\end{align}
	
		\noi
	Note that the ``loss'' $\eps_1$ comes from the fact that 
	$b = \frac{1}{2}+\eps_1$ is greater than $ \frac{1}{2}$
	and $\eps_2$ comes from Lemma 	\ref{LEM:timedecay}.		
	Therefore, in proving local well-posedness via the fixed point argument, we require 
	\begin{equation}
		\label{localtime} \dl^{\frac{1}{2}-\eps_1-\eps_2} \| u \|_{X^{0, \frac{1}{2}+, \dl}}^2 \lesssim 1 
	\end{equation}
	
	\noi on the ball $\{ u: \| u \|_{X^{0, \frac{1}{2}+, \dl}} \leq 2C_0\|u_0\|_{L^2}\}$. Hence, we can choose $\dl \sim \|u_0\|_{L^2}^{-(4+\theta)}$ with $\theta = 0+$.
	
	Let $T>0$ and $\eps>0$ be given, and we continue the argument from Subsection \ref{SUBSEC:1GWP}. First, in view of the large deviation estimate \eqref{largedevi}, choose $K \sim \big(\log \frac{1}{\eps}\big)^\frac{1}{2}$ so that $\mathbb{P} ( \| u_0(\omega)\|_{H^s} \geq K ) \leq \frac{1}{2}\eps.$ In the following, we assume $\| u_0\|_{H^s} \leq K$.
	Now, fix $\dl \sim N^{(4+\theta)s} K^{-(4+\theta)}$ with $\theta = 0+$.
For fixed $\al \leq \frac{1}{2}$,  $ s= \al - \frac{1}{2}- < 0$ is also fixed.
Hence, we can write
\begin{equation}
	\label{DL} \dl \sim N^{4s - } K^{-4-}, 
\end{equation}
	
	\smallskip
	
	Before proceeding further, we present an important proposition whose proof is given in the remaining part of the paper. 
	\begin{proposition}
		\label{PROP:HNLS} Let $ s= \al - \frac{1}{2}-$ with $\al \in ( \frac{5}{12}, \frac{1}{2}]$. Given $T> 0$ and $K> 0$, there exists $N$ sufficiently large with $\dl \sim N^{4s-} K^{-4-}$ such that the following holds. Suppose that 
		\begin{equation}
			\label{ujbound} \|u^j(t)\|_{X^{0, \frac{1}{2}+}[(j-1)\dl, j\dl]} \leq C N^{-s}K 
		\end{equation}
		
		\noi such that $\dl^{\frac{1}{2}-}\|u^j(t)\|_{X^{0, \frac{1}{2}+}[(j-1)\dl, j\dl]}^2 \lesssim 1$ (see \eqref{localtime}) for $j = 1, \cdots, [\frac{T}{\dl}]$. Write the solution $v^j$ of the following difference equation: 
		\begin{equation}
			\label{HNLSj} 
			\begin{cases}
				i \dt v^j - \dx^2 v^j \pm (\mathcal{N} (u^j + v^j) - \mathcal{N}(u^j)) = 0 \\
				v^j|_{t= (j-1)\dl} = \psi_{j-1} := \sum_{|n|> N} \frac{g_n(\omega)e^{i (j - 1) \dl n^2}}{\sqrt{1+|n|^{2\al}}} e^{inx} 
			\end{cases}
		\end{equation}
		
		\noi as $v^j (t) = S(t-(j-1)\dl)\psi_{j-1} + w^j(t)$. Then, \eqref{HNLSj} is locally well-posed on the time interval $[(j-1) \dl, j\dl]$ except on a set of measure $e^{-\frac{1}{\dl^c}}$ for each $j = 1, \cdots, [\frac{T}{\dl}]$. Moreover, we have the following bound on the nonlinear terms: 
		\begin{equation}
			\label{wjbound} \sum_{j = 1}^{[T/\dl]} \|w^j(j\dl)\|_{L^2} \lesssim N^{-s}K. 
		\end{equation}
	\end{proposition}

\begin{remark} \rm
In Proposition \ref{PROP:HNLS}, we do not assume that $u^j$ is deterministic.
In our application, $u^j$ is indeed random -- not even independent from $\psi^{j-1}$ and $v^j$.
\end{remark}
	
	Now, we continue the proof of Proposition \ref{THM:GWP2}. Our choice of $\dl$ guarantees that \eqref{LNLS1} is well-posed on $[0, \dl]$ with the bound \eqref{u1bound}. Then, by Proposition \ref{PROP:HNLS}, \eqref{HNLS1} is well-posed on $[0, \dl]$ except on a set of measure $e^{-\frac{1}{\dl^c}}$ with the bound \eqref{w1bound}, which in turn shows that \eqref{LNLS2} is well-posed on $[\dl, 2\dl]$ with the bound \eqref{u2bound}. 
	
	Write the solution $v^2$ of \eqref{HNLS2} as $v^2 (t) = S(t-\dl)\psi_1 + w^2(t)$. It follows from \eqref{u2bound} and Proposition \ref{PROP:HNLS} that \eqref{HNLS2} is well-posed on the time interval $[\dl, 2\dl]$ except on a set of measure $e^{-\frac{1}{\dl^c}}$. Moreover, we have 
	\begin{equation}
		\label{w2bound} \sum_{j = 1}^2 \|w^j(j\dl)\|_{L^2} \lesssim N^{-s}K. 
	\end{equation}
	
	At time $t = 2\dl$, write $u (2\dl) = \phi_2 + \psi_2$, where $\phi_2 := u^2(2\dl) + w^2(2\dl)$ and $\psi_2 := S(\dl) \psi_1 = S(2\dl) \psi_0$. Then, \eqref{w2bound} guarantees that the solution $u^3$ to 
	\begin{equation}
		\label{LNLSj} 
		\begin{cases}
			i \dt u^j - \dx^2 u^j \pm \mathcal{N}(u^j) = 0 \\
			u^j|_{t= (j-1) \dl} = \phi_{j-1} 
		\end{cases}
	\end{equation}
	
	\noi with $j = 3$ satisfies 
	\begin{equation}
		\label{u3bound} \| u^3 \|_{X^{0, \frac{1}{2}+}[2\dl, 3\dl]} \leq \| \phi_0 \|_{L^2} + \sum_{j = 1}^2 \|w^j(j\dl)\|_{L^2} \lesssim N^{-s}K. 
	\end{equation}
	
	Clearly, we can iterate this argument to show that \eqref{NLS2} is well-posed on $[0, T]$, assuming \eqref{wjbound}. Lastly, note that the measure of the exceptional sets can be estimated by 
	\begin{align*}
		\Big[\frac{T}{\dl}\Big] e^{-\frac{1}{\dl^c}} \leq e^{\ln \frac{T}{\dl} -\frac{1}{\dl^c}} \leq e^{ -\frac{1}{2}\frac{1}{\dl^c}} < \frac{1}{2}\eps 
	\end{align*}
	
	\noi for sufficiently small $\dl >0$, i.e. for sufficiently large $N = N(T, \eps)$. This completes the proof of Proposition \ref{THM:GWP2}. 
\end{proof}

\subsection{Basic Setup} \label{SUBSEC:GWP2}

In the remaining part of the paper, we prove Proposition \ref{PROP:HNLS}. In the following, fix $T> 0 $ and $K > 0$, and let $ s= \al - \frac{1}{2}-$ and \eqref{DL}:
\begin{equation*}
	\dl \sim N^{4s-} K^{-4-}, 
\end{equation*}

\noi where $N = N(T, K)$ to be determined later.

First, consider the following difference equation: 
\begin{equation}
	\label{HNLSjj} 
	\begin{cases}
		i \dt v - \dx^2 v \pm (\mathcal{N} (u^0 + v) - \mathcal{N}(u^0)) = 0 \\
		v|_{t= 0} = \psi = \sum_{|n|> N} \frac{c_n g_n(\omega)}{\sqrt{1+|n|^{2\al}}} e^{inx} 
	\end{cases}
\end{equation}

\noi where $|c_n| = 1$ for all $n\in \mathbb{Z}$ and $u^0$ is a given function such that 
\begin{equation}
	\label{ujjbound} \|u^0(t)\|_{X^{0, \frac{1}{2}+, \dl}} \leq C N^{-s}K 
\end{equation}

\noi satisfying \eqref{localtime}. Let $w$ denote the nonlinear part of the solution $v$ of \eqref{HNLSjj}. i.e. it is given by 
\begin{equation}
	\label{wjj} w (t) := w(t; v, \psi, u^0) = \pm i \int_{0}^t S(t - t')\big(\mathcal{N} (u^0 + v) - \mathcal{N}(u^0)\big)(t') dt' 
\end{equation}

\noi for $ t \in [ 0, \dl]$. From the linear estimate \eqref{duhamel}, 
we have 
\begin{equation} \label{duhamel2}
\|w(\dl)\|_{L^2} \lesssim \|\eta_{_\dl}(t) w\|_{X^{0, \frac{1}{2}+\eps_1, \dl}}
\lesssim \| \mathcal{N} (\wt{u^0} + \wt{v}) - \mathcal{N}(\wt{u^0})\|_{X^{0, -\frac{1}{2}+\eps_1}},
\end{equation}

\noi
for some small $\eps_1>0$, where $\wt{u^0}$ and $ \wt{v}$ are extensions of $u^0$ and $v$, respectively.

{\it Suppose} that we have 
\begin{equation}
	\label{size1} \| \mathcal{N} (\wt{u^0} + \wt{v}) - \mathcal{N}(\wt{u^0})\|_{X^{0, -\frac{1}{2}+}} \lesssim N^{3s - \g} 
\end{equation}

\noi with some small $\g >0$ except on a set of measure $e^{-\frac{1}{\dl^c}}$
(for any extensions $\wt{u^0}$ and $ \wt{v}$  of $u^0$ and $v$.) 
Then, it follows that the mapping $\G$ defined by 
\begin{equation}
	\label{GNLS} \G v(t) := S(t) \psi +w(t; v, \psi, u^0) 
\end{equation}

\noi is a contraction on $S(t) \psi^\omega + B$ on the time interval $[0, \dl]$ except on a set of measure $e^{-\frac{1}{\dl^c}}$, where $B$ denotes the ball of radius $\sim N^{3s-\g }$ in $X^{0, \frac{1}{2}+}_{[0,  \dl]}$. Moreover, from \eqref{DL}, \eqref{duhamel2},  and \eqref{size1}, we have 
\begin{align}
	\label{wjjbound} \frac{T}{\dl} \|w(\dl)\|_{L^2} \lesssim T \dl^{-1} N^{3s-\g} \lesssim T K^{4+} N^{-s}N^{ -\g+ } \lesssim N^{-s}K 
\end{align}

\noi for sufficiently large $N = N(T, K)$. Note that \eqref{duhamel2} and \eqref{size1} imply only the boundedness of the map $\G$ from $S(t) \psi^\omega + B$ into itself. In establishing the contraction property, 
one needs to consider the difference $\G v_1 - \G v_2$ for $v_1, v_2 \in S(t) \psi^\omega + B$.
We omit details.

Finally, note that the bound \eqref{ujbound} on $u^j$ is uniform in $j$ in Proposition \ref{PROP:HNLS}. 
Hence, the above local well-posedness result can be applied to \eqref{HNLSj} on $[(j - 1)\dl, j\dl]$ for $j = 1, \cdots, [\frac{T}{\dl}]$, and moreover \eqref{wjbound} follows from \eqref{wjjbound}. Therefore, it remains to prove \eqref{size1} for $\al \in ( \frac{5}{12}, \frac{1}{2}]$ (and for large $N$.)

\medskip

Note that  \eqref{size1} follows, once we prove 
\begin{equation}
	\label{size2} \| \mathcal{N}_j (u_1, u_2, u_3) \|_{X^{0, -\frac{1}{2}+\eps_1}} \lesssim N^{3s - \g}, \ j = 1, 2, 
\end{equation}

\noi except on a set of measure $e^{-\frac{1}{\dl^c}}$, where 
$\N_j$ is as in \eqref{NN1} or \eqref{NN2}, and $u_j$ is either of type 
\begin{itemize}
	\item[(I)] linear part: random, less regular
	\[u_j (x, t) = S(t) \psi= \sum_{|n| > N } \frac{c_n g_n(\omega)}{\sqrt{1+|n|^{2\al}}} e^{i(nx + n^2t)} \text{ with } |c_n| = 1, \text{ or }\]
	\item[(II)] smoother:
	\begin{itemize}
	\item[(II.a)] ``high frequency'' nonlinear part: small
	\begin{equation}
			\label{IIa} 
		\ u_j = \wt{w}, \text{ where $\wt{w}$ is an extension of $w$} \text{ with } 
		\|w \|_{X^{0, \frac{1}{2}+\eps_1, \dl}} \lesssim N^{3s-\g},
	\end{equation}
	\item[(II.b)] ``low frequency'' input: large
	\begin{align} \label{IIb} 
			 u_j = \wt{u^0}, & \text{ where $\wt{u^0}$ is an extension of $u^0$}
			 \text{ with } \|u^0 \|_{X^{0, \frac{1}{2}+\eps_1, \dl}} \lesssim N^{-s}K  \\ 
			 & 			 \text{ satisfying } \eqref{localtime}:\dl^{\frac{1}{2}-\eps_1 -\eps_2} 
			 \| u^0 \|_{X^{0, \frac{1}{2}+\eps_1, \dl}}^2 \lesssim 1,  \notag 
	\end{align}
		\end{itemize}
\end{itemize}

\noi {\it except} for the case $u_j = u^0$ for all $j= 1, 2, 3$. We may insert the smooth cutoff function $\eta_{_\dl}$ supported on $[-2\dl, 2\dl]$ if necessary.

Note that $u^0$ has a much larger norm than $w$ since $s <0$. 
Thus, without loss of generality, 
we assume that $u_j = u^0$ if $u_j$ is of type (II), unless $u_j$ is of type (II) for {\it all} $j= 1, 2, 3$. In the latter case, we may assume that two of $u_j$'s are $u^0$ and the remaining $u_j$ is $w$., and it suffices to prove 
\begin{equation}
	\label{size3} \| \mathcal{N}_j (u^0, u^0, w) \|_{X^{0, -\frac{1}{2}+\eps_1}} \lesssim N^{3s - \g}, \ j = 1, 2, 
\end{equation}

\noi {\it assuming}  \eqref{localtime}. In the following subsections, we prove \eqref{size2} by separately estimating the contributions from $\mathcal{N}_1$ and $\mathcal{N}_2$. 
Indeed, except for Case (A) in Subsection \ref{SUBSEC:GWP4}
(namely with $u^0$, $u^0$, and $w$),
we can prove \eqref{size2} with $N^{3s-\g-}$ instead of $N^{3s - \g}$.

In the following,  we write estimates 
directly with 
$\|u_j\|_{X^{s, b, \dl}}$  for simplicity of presentation,
meaning that the same estimates hold with $\|\wt{u}_j\|_{X^{s, b}}$ 
for any extension $\wt{u}_j$ of $u_j$
(and thus we can take the infimum over $\wt{u}_j$.)
See Remark \ref{REM:local}.

\subsection{Estimate on $\mathcal{N}_2$} \label{SUBSEC:GWP3}

In this subsection, we prove the estimate \eqref{size2} for $\mathcal{N}_2 (u_1, u_2, u_3)$ defined in \eqref{NN2}. 
We need to estimate
\begin{align}
	\label{easy1b}
	\| \mathcal{N}_2 (u_1, u_2, u_3 & ) \|_{X^{0, -\frac{1}{2}+}} 
	= \bigg\| \frac{1}{\jb{\tau - n^2}^{\frac{1}{2}-}} \intt_{\tau = \tau_1 - \tau_2 + \tau_3} \ft{u}_1(n, \tau_1)\cj{\ft{u}_2(n, \tau_2)}\ft{u}_3(n, \tau_3) d\tau_1 d\tau_2 \bigg\|_{l^2_n L^2_\tau} \notag \\
	\intertext{By H\"older inequality with $p$ large ($\frac{1}{2} =\frac{1}{2+} + \frac{1}{p}$), } & \lesssim \sup_n \|\jb{\tau - n^2}^{-\frac{1}{2}+}\|_{L^{2+}_\tau} \Big\| \intt_{\tau = \tau_1 - \tau_2 + \tau_3} \ft{u}_1(n, \tau_1)\cj{\ft{u}_2(n, \tau_2)}\ft{u}_3(n, \tau_3) d\tau_1 d\tau_2 \Big\|_{l^2_{n} L^p_\tau} . 
\end{align}

\noi In the following, we omit details if the computation is basically the same as in 
Subsection \ref{SUBSEC:LWP3}.
Recall $\dl \sim N^{4s-}K^{-4-}$
from \eqref{DL}.
We assume that $N$ is sufficiently large in the following.

\medskip

\noi $\bullet$ {\bf Case (a):} $u_j$ of type (II), $j = 1, \dots, 3$.

In this case, we prove \eqref{size3}. By Young and H\"older inequalities in $\tau$, followed by H\"older in $n$, $l^2_n \subset l^6_n$, and Lemma \ref{LEM:timedecay}, we have 
\begin{align*}
	\eqref{easy1b} & \lesssim \prod_{j = 1}^3 \| \jb{\tau - n^2}^{\frac{1}{6}+} \ft{u}_j (n, \tau) \|_{l^6_n L^2_\tau} \leq \dl^{1-} \| u^0 \|_{X^{0, \frac{1}{2}+, \dl}}^2 \| w \|_{X^{0, \frac{1}{2}+, \dl}} \\
	& \lesssim \dl^{\frac{1}{2}-} N^{3s - \g } \lesssim N^{3s-\g-} 
\end{align*}

\noi for $s \leq 0 $.

\noi $\bullet$ {\bf Case (b):} $u_j$ of type (I), $j = 1, \dots, 3$.

By Lemma \ref{LEM:prob1} and  Young's inequality, we have
\begin{align*}
	\eqref{easy1b} & \lesssim \dl^{1-} \|\jb{n}^{-3\al} |g_n (\omega)|^3 \|_{l^2_{|n| > N}} \lesssim \dl^{1 -\frac{3}{2}\beta-} \|\jb{n}^{-3\al+ 3\eps} \|_{l^2_{|n| > N}} \\
	& \lesssim \dl^{1 -\frac{3}{2}\beta-} N^{-3\al+ \frac{1}{2}+ 3\eps} \lesssim N^{3s-2\al + } K^{-4-} \lesssim N^{3s -\g-} 
\end{align*}

\noi for $ \al > \frac{1}{2}\g > 0$
outside an exceptional set of measure $< e^{-\frac{1}{\dl^{c}}}$.

\noi $\bullet$ {\bf Case (c):} Exactly two $u_j$'s of type (I). Say $u_1(\I)$, $u_2(\I)$, and $u_3(\II)$. 

By Young's inequality and Lemmata \ref{LEM:timedecay} and \ref{LEM:prob1}, we have 
\begin{align*}
	\eqref{easy1b} & \lesssim \dl^{\frac{1}{2}-} \big(\sup_{|n|>N} \jb{n}^{-2\al} |g_n|^2\big) \big\| \ft{u^0} (n, \tau) \big\|_{l^2_n L^2_\tau} \lesssim \dl^{1 -\beta-} N^{-2\al+ 2\eps}\|u^0 \|_{X^{0, \frac{1}{2}+, \dl}}\\
	&\lesssim N^{3s-2\al + } K^{-3-} \lesssim N^{3s -\g-} 
\end{align*}

\noi for $ \al > \frac{1}{2}\g > 0$ outside an exceptional set of measure $< e^{-\frac{1}{\dl^c}}$.

\noi $\bullet$ {\bf Case (d):} Exactly one $u_j$ of type (I). Say $u_1(\I)$, $u_2(\II)$, and $u_3(\II)$.

By Young's inequality, followed by H\"older inequality in $n$ ($\frac{1}{2} = \frac{1}{4}+ \frac{1}{4}$) and in $\tau$ ($\frac{3}{4} = \frac{1}{2} + \frac{1}{4}$) and Lemmata \ref{LEM:timedecay} and \ref{LEM:prob1}, we have 
\begin{align*}
	\eqref{easy1b} & \lesssim \dl^{\frac{1}{2}-} \big(\sup_{|n| >N} \jb{n}^{-\al} |g_n|\big) \Big\| \big\| \ft{u^0} (n, \tau)\big\|_{L^\frac{4}{3}_\tau}^2 \Big\|_{l^2_{n}}\\
	& \lesssim \dl^{\frac{1}{2}-\frac{\beta}{2} - } N^{-\al+\eps} \sup_n \|\jb{\tau- n^2}^{-\frac{1}{4}-}\|^2_{L^4_\tau} \| \jb{\tau- n^2}^{\frac{1}{4}+} \ft{u^0} (n, \tau) \|^2_{l^4_n L^2_\tau} \\
	& \lesssim \dl^{1-\frac{\beta}{2} - } N^{-\al+\eps} \| \jb{\tau- n^2}^\frac{1}{2} \ft{v}_j (n, \tau) \|_{l^4_n L^2_\tau}^2 \lesssim N^{2s-\al+}K^{-2-} \lesssim N^{3s -\g-} 
\end{align*}

\noi for $\al > \frac{1}{4} + \frac{1}{2}\g >\frac{1}{4}$ outside an exceptional set of measure $< e^{-\frac{1}{\dl^c}}$. 

\subsection{Estimate on $\mathcal{N}_1$: High Modulation Cases} \label{SUBSEC:GWP4}

In the next two subsections, we prove the main part of the estimate \eqref{size2}: 
\begin{equation}
	\label{trilinear3b} \| \mathcal{N}_1(u_1, u_2, u_3)\|_{X^{0, -\frac{1}{2}+\eps_1, \dl}} \lesssim N^{3s -\g} 
\end{equation}

\noi for some small $ \g > 0$, 
where $\N_1$ is as in \eqref{NN1} and $u_j$ is of type (I) or (II). Once again, we omit details in the following when the computation basically follows from Subsection \ref{SUBSEC:LWP4}.

Using duality, we can estimate \eqref{trilinear3b} by 
\begin{equation}
	\label{duality1b} \iint u^1 u^2 u^3 \cdot v \, dx dt 
\end{equation}

\noi where $\| v\|_{X^{0, \frac{1}{2}-\eps_1, \dl}} \leq 1$ (with the complex conjugate on an appropriate $u^j$.)
We assume that $N$ is sufficiently large in the following.

\medskip

\noi $\bullet$ {\bf Case (A):} $u^1$ and $u^2$ are of type $(\II)$.

Suppose that $u^3$ is of type $(\II)$. In this case, we prove \eqref{size3} instead of \eqref{trilinear3b}. 
By Lemmata \ref{LEM:deterministic} (a) and \ref{LEM:timedecay} with \eqref{IIa} and \eqref{IIb}
(also see \eqref{localtime}), we have 
\begin{align*}
	\eqref{duality1b} 
	& \lesssim \dl^{\frac{1}{2}-\eps_1-\eps_2} \| u^0\|^2_{X^{0, \frac{1}{2}, \dl}} \| w\|_{X^{0, \frac{1}{2}, \dl}} 
	\| v\|_{X^{0, \frac{1}{2}-\eps_1, \dl}}\\
	& \lesssim   \| w\|_{X^{0, \frac{1}{2}, \dl}} 
		\lesssim N^{3s-\g}. 
\end{align*}

Next, suppose that $u^3$ is of type $(\I)$ i.e. $u^3 = S(t) u_0$.
In this case, we do not need to apply dyadic decomposition on $u^1$ and $u^2$.
Namely, for a fixed dyadic block $N^3$ for $u^3$ of type $(\I)$, 
with a slight abuse of notation, 
we  use $u^1$ and $u^2$ to denote the sums of $u^j$ over the dyadic blocks $N^j \geq N^3$, $j = 1, 2$.

By Lemma \ref{LEM:deterministic} (b) with $p$ large followed by Lemma \ref{LEM:prob2}, we have 
\begin{align*}
			\eqref{duality1b} 
			\lesssim (N^3)^{\frac{1}{2}-\al+} \| u^0\|^2_{X^{0, \frac{1}{4}+, \dl}} \| v\|_{X^{0, \frac{1}{4}+, \dl}} 
\end{align*}

\noi outside an exceptional set of measure $< e^{-\frac{1}{\dl^c}}$. If $\jb{\tau_j - n_j^2}^{\frac{1}{4}-} \gtrsim (N^3)^{\frac{1}{2}-\al+} N^{-3s+\g+\wt{\eps}}$ for $u_j$ of type $(\II)$, or if $\jb{\tau - n^2}^{\frac{1}{4}-} \gtrsim (N^3)^{\frac{1}{2}-\al+} N^{-3s+\g+\wt{\eps}}$, then it follows 
from Lemma \ref{LEM:timedecay}, \eqref{DL}, and \eqref{IIb} that 
\begin{align*}
	\eqref{duality1b} \lesssim \dl^{\frac{1}{2}-} N^{-2s} K^2 N^{3s-\g-\wt{\eps}} \lesssim N^{3s-\g-} 
\end{align*}

\noi for $N$ sufficiently large. Recall $N^3 > N$, $s = \al - \frac{1}{2}-$, and $\g = 0+$. 

Hence, it remains to estimate the contribution to \eqref{trilinear3b}
from the region satisfying
\begin{equation}
	\label{case0b} \jb{\tau - n^2} \ll (N^3)^{8-16\al+}, \text{ and } \jb{\tau_j - n_j^2} \ll (N^3)^{8-16\al+} \text{ if } u_j \text{ of type }(\II) 
\end{equation}
in the following.

\medskip

\noi $\bullet$ {\bf Case (B):} $u^1$ of type $(\II)$, and $u^2$ of type $(\I)$.

In this case, we do not need to apply dyadic decomposition on $u^1$.
Namely, for a fixed dyadic block $N^2$ for $u^2$ of type $(\I)$, 
we use $u^1$ to denote the sum of $u^1$ over the dyadic blocks $N^1 \geq N^2$.

\noi $\circ$ Subcase (B.1): $u^3$ is of type $(\II)$. 
By Lemma \ref{LEM:deterministic} (b) with $p$ large followed by Lemma \ref{LEM:prob2}, we have 
\begin{align*}
	\eqref{duality1b} 
	\lesssim (N^2)^{\frac{1}{2}-\al+}\|u^0\|^2_{X^{0, \frac{1}{4}+, \dl}} \|v\|_{X^{0, \frac{1}{4}+, \dl}} 
\end{align*}

\noi outside an exceptional set of size $<e^{-\frac{1}{\dl^c}}$. If $\jb{\tau_j - n_j^2}^{\frac{1}{4}-} \gtrsim (N^2)^{\frac{1}{2}-\al+}N^{-3s+\g+}$ for $u_j$ of type $(\II)$, or if $\jb{\tau - n^2}^{\frac{1}{4}-} \gtrsim (N^2)^{\frac{1}{2}-\al+}N^{-3s+\g+}$, then \eqref{trilinear3b} follows as in Case (A). 

Hence, it remains to estimate the contribution to \eqref{trilinear3b}
from the region satisfying
\begin{equation}
	\label{caseAb} \jb{\tau - n^2} \ll (N^2)^{8-16\al+}, \text{ and } \jb{\tau_j - n_j^2} \ll (N^2)^{8-16\al+} \text{ if } u_j \text{ of type }(\II). 
\end{equation}

\noi $\circ$ Subcase (B.2): $u^3$ is of type $(\I)$. 
By Lemma \ref{LEM:deterministic} (b) with $p$ large followed by Lemma \ref{LEM:prob2}, we have 
\begin{align*}
	\eqref{duality1b} 
	\lesssim (N^2)^{1-2\al+}\|u^0\|_{X^{0, 0+, \dl}}\|v\|_{X^{0, 0+, \dl}} 
\end{align*}

\noi outside an exceptional set of measure $<e^{-\frac{1}{\dl^c}}$. If $(\s^1)^{\frac{1}{2}-} \gtrsim (N^2)^{1-2\al+}N^{-2s+\g+}$ or if $\jb{\tau - n^2}^{\frac{1}{2}-} \gtrsim (N^2)^{1-2\al+}N^{-2s+\g+}$, 
then \eqref{trilinear3b} follows from Lemma \ref{LEM:timedecay}, \eqref{DL}, and \eqref{IIb}. 

Hence, it remains to estimate the contribution to \eqref{trilinear3b}
from the region satisfying
\begin{equation}
	\label{caseB2} \jb{\tau - n^2} \ll (N^2)^{4-8\al+}, \text{ and } \jb{\tau_j - n_j^2} \ll (N^2)^{4-8\al+} \text{ if } u_j \text{ of type }(\II). 
\end{equation}

\medskip

\noi $\bullet$ {\bf Case (C):} $u^1$ of type $(\I)$, and $u^2$, $u^3$ of type $(\II)$.

Dyadically decompose all the spatial frequencies. Suppose $\jb{\tau - n^2} \gg \max(\s^2, \s^3)$. 
By Lemma \ref{LEM:deterministic} (c) with $p$ large, Lemmata \ref{LEM:prob2} and \ref{LEM:timedecay},
and \eqref{IIb}, we have 
\begin{align*}
	\eqref{duality1b}  
	 \lesssim (N^1)^{\frac{1}{2}-\al+}\|u^0\|_{X^{0, \frac{3}{8}, \dl}}^2\|v\|_{X^{0, 0+, \dl}} 
	\lesssim \dl^{\frac{1}{4}-} (N^1)^{\frac{1}{2}-\al+} N^{-2s}K^2 \|v\|_{X^{0, 0+}} 
\end{align*}

\noi outside an exceptional set of measure $<e^{-\frac{1}{\dl^c}}$. 
Hence, as before, \eqref{trilinear3b} follows as long as $\jb{\tau - n^2}^{\frac{1}{2}-} \gtrsim (N^1)^{ \frac{1}{2} - \al+}N^{-4s+\g+}$. Similar results hold if $\s^2 \gg \max(\s^3, \jb{\tau-n^2})$ or $\s^3 \gtrsim \max(\s^2, \jb{\tau-n^2})$. 

Hence, it remains to estimate the contribution to \eqref{trilinear3b}
from the region satisfying
\begin{equation}
	\label{caseBb} \jb{\tau - n^2} \ll (N^1)^{5 - 10\al+}, \text{ and } \jb{\tau_j - n_j^2} \ll (N^1)^{5 - 10\al+} \text{ if } u_j \text{ of type }(\II). 
\end{equation}

\medskip

\noi $\bullet$ {\bf Case (D):} $u^1$ of type $(\I)$, and either $u^2(\I)$, $u^3(\II)$ or $u^2(\II)$, $u^3(\I)$.

Suppose that $u^2$ is of type $(\I)$ and that $u^3$ is of type $(\II)$. Moreover, suppose $\jb{\tau - n^2} \gg \s^3$. 
By Lemma \ref{LEM:deterministic} (c) with $p$ large, 
 Lemmata \ref{LEM:prob2} and \ref{LEM:timedecay}, and \eqref{IIb}, we have 
\begin{align*}
	\eqref{duality1b} 
	 \lesssim (N^1)^{1-2\al+}\|u^0\|_{X^{0, 0+}}\|v\|_{X^{0, 0}} 
	\lesssim \dl^{\frac{1}{2}-} (N^1)^{1-2\al+}N^{-s}K\|v\|_{X^{0, 0}} 
\end{align*}

\noi outside an exceptional set of measure $<e^{-\frac{1}{\dl^c}}$. Hence, \eqref{trilinear3b} follows as long as $\jb{\tau - n^2}^{\frac{1}{2}-} \gtrsim (N^1)^{1 - 2\al+}N^{-2s +\g+}$. Similar results hold if $\s^3 \gtrsim \jb{\tau-n^2}$, (or $u^2$ is of type $(\II)$ and $u^3$ is of type $(\I)$.) 

Hence, it remains to estimate the contribution to \eqref{trilinear3b}
from the region satisfying
\begin{equation}
	\label{caseCb} \jb{\tau - n^2} \ll (N^1)^{4 - 8\al+}, \text{ and } \jb{\tau_j - n_j^2} \ll (N^1)^{4 - 8\al+} \text{ if } u_j \text{ of type }(\II). 
\end{equation}

\medskip

\noi {\bf Summary:} 
By repeating the computation in Subsection \ref{SUBSEC:LWP4},
we can reduce the estimate into the following two cases (with $\theta = 0+$):.

\medskip

\noi $\bullet$ $u^1$ is of type $(\II)$: By \eqref{v1} and \eqref{v2}, we can bound \eqref{trilinear3b} as follows: 
\begin{align}
	\label{u2b} \eqref{trilinear3b} \lesssim \dl^\theta 
	M(N, N^2, N^3) \bigg( \sum_n \Big| \sum_{\substack{ n = n_1 - n_2 + n_3 \\
	n_2 \ne n_1, n_3\\
	n^2 = n_1^2 - n_2^2 + n_3^2 + \mu}} a_1(n_1)\cj{a_2(n_2)}a_3(n_3) \Big|^2 \bigg)^\frac{1}{2}, 
\end{align}

\noi where $\sum_{n} |a^1(n)|^2 \leq 1$, $a^2(n) = \frac{g_{n}(\omega)}{1 + |n|^\al}$, $a^3(n) = \frac{g_{n}(\omega)}{1 + |n|^\al}$ or $\sum_{|n| \sim N^3} |a^3(n)|^2 \leq 1$, and 
\begin{tabbing}
	\hspace{1cm} \=Case (A): \hspace{1cm}\=$M(N, N^2, N^3) = (N^3)^{0+} N^{-2s} $ \=and \=$|\mu| \ll (N^3)^{8-16\al+}$ \\
	
	\>Subcase (B.1): \>$M(N, N^2, N^3) = (N^2)^{0+} N^{-2s}$ \>and \>$|\mu| \ll (N^2)^{8-16\al+}$ \\
	
	\>Subcase (B.2): \>$M(N, N^2, N^3) = (N^2)^{0+} N^{-s}$ \>and \>$|\mu| \ll (N^2)^{4-8\al+}$. 
\end{tabbing}
Note that we did not apply dyadic decomposition on $N^1$.

\medskip 

\noi $\bullet$ $u^1$ is of type $(\I)$: By \eqref{v1} and \eqref{v2}, we can bound \eqref{trilinear3b} as follows: 
\begin{align}
	\label{u1b} \eqref{trilinear3b} \lesssim \dl^\theta (N^1)^{0+} M(N) \bigg( \sum_{|n| \lesssim N^1} \Big| \sum_{\substack{ n = n_1 - n_2 + n_3 \\
	n_2 \ne n_1, n_3\\
	n^2 = n_1^2 - n_2^2 + n_3^2 + \mu}} a_1(n_1)\cj{a_2(n_2)}a_3(n_3) \Big|^2 \bigg)^\frac{1}{2}, 
\end{align}

\noi where $a^1(n) = \frac{g_{n}(\omega)}{1 + |n|^\al}$, $a^j(n) = \frac{g_{n}(\omega)}{1 + |n|^\al}$ or $\sum_{|n| \sim N^j} |a^j(n)|^2 \leq 1$ for $j = 2, 3$, and 
\begin{tabbing}
	\hspace{1cm} \=Case (C): \hspace{1cm}\=$M(N) = N^{-2s} $ \=and \=$|\mu| \ll (N^1)^{5-10\al+}$ \\
	
	\>Case (D): \>$M(N) = N^{-s}$ \>and \>$|\mu| \ll (N^1)^{4-8\al+}$\\
	
	\>All type (\I): \>$M(N) = 1$ \>and \>$|\mu| \lesssim (N^1)^2$.
\end{tabbing}
Note that all the spatial frequencies are dyadically decomposed.

\medskip

By symmetry between $u_1$ and $u_3$, we assume $|n_1| \sim N^1$ or $|n_2| \sim N^1$ in the following. Moreover, in Subcase (B.2) and Case (D), we may assume that $|n_1| \sim N^1$. If not, say, we have $|n_2| > 10 (|n_1| + |n_3|)$. Then, $|\mu| \sim |(n_2 - n_1) (n_2 - n_3)| \sim |n_2|^2 \sim (N^1)^2$ by \eqref{ALGEBRA}. In these two cases, we have $|\mu| \ll (N^j)^{4- 8\al+} \ll (N^1)^2$ as long as $\al > \frac{1}{4}$. i.e. we would have a contradiction.

Lastly, we list all the different cases as before.
We consider these cases in details in the next subsection.

\noi $\bullet$ $n_1 = N^1$: 
\begin{tabbing}
	\hspace{7mm} \=Case (a): \= $n_1 = N^1(\II)$, \= $n_2 = N^2(\I)$, \= $n_3 = N^3(\II)$ or \= $n_2 = N^3(\I)$, \= $n_3 = N^2(\II)$ \\
	
	\>Case (b): \>$n_1 = N^1(\II)$, \>$n_2 = N^3(\II)$, \> $n_3 = N^2(\I)$ or \>$n_2 = N^2(\II)$, \> $n_3 = N^3(\I)$ \\
	
	\>Case (c): \> $n_1 = N^1(\I)$, \>$n_2 = N^2(\II)$, \>$n_3 = N^3(\II)$ \\
	
	\>Case (d): \>$n_1 = N^1(\I)$, \>$n_2 = N^3(\II)$, \>$n_3 = N^2(\II)$\\
	
	\>Case (e): \>$n_1 = N^1(\II)$, \>$n_2 = N^2(\I)$, \>$n_3 = N^3(\I)$\\
	
	\>Case (f): \>$n_1 = N^1(\II)$, \>$n_2 = N^3(\I)$, \>$n_3 = N^2(\I)$\\
	
	\>Case (g): \>$n_1 = N^1(\I)$, \>$n_2 = N^2(\I)$, \>$n_3 = N^3(\II)$\\
	
	\>Case (h): \>$n_1 = N^1(\I)$, \>$n_2 = N^3(\I)$, \>$n_3 = N^2(\II)$\\
	
	\>Case (i): \>$n_1 = N^1(\I)$, \>$n_2 = N^2(\II)$, \>$n_3 = N^3(\I)$\\
	
	\>Case (j): \>$n_1 = N^1(\I)$, \>$n_2 = N^3(\II)$, \>$n_3 = N^2(\I)$\\
	
	\>Case (k): \>All type (\I)\\
	
\end{tabbing}

\noi $\bullet$ $n_2 = N^1$: 
\begin{tabbing}
	\hspace{7mm} \=Case (a'): \= $n_2 = N^1(\II)$, \= $n_1 = N^2(\I)$, \= $n_3 = N^3(\II)$ or \= $n_1 = N^3(\I)$, \= $n_3 = N^2(\II)$ \\
	
	\>Case (b'): \>$n_2 = N^1(\II)$, \>$n_1 = N^3(\II)$, \> $n_3 = N^2(\I)$ or \>$n_1 = N^2(\II)$, \> $n_3 = N^3(\I)$ \\
	
	\>Case (c'): \> $n_2 = N^1(\I)$, \>$n_1 = N^2(\II)$, \>$n_3 = N^3(\II)$ \\
	
	\>Case (d'): \>$n_2 = N^1(\I)$, \>$n_1 = N^3(\II)$, \>$n_3 = N^2(\II)$\\
	
	\>Case (k'): \>All type (\I)
\end{tabbing}

\subsection{Estimate on $\mathcal{N}_1$: Low Modulation Cases} \label{SUBSEC:GWP5}

As before, we use $|n|^\al $ for $1+ |n|^\al$ and drop a complex conjugate on $u_2$ when it plays no significant role. 
Let $A_n$ and $B_n$ be as in Subsection \ref{SUBSEC:LWP4}.
Recall \[\mu = 2 (n_2 - n_1) (n_2 - n_3) = 2 (n - n_1) (n - n_3) \] 

\noi
from \eqref{ALGEBRA7}, $s = \al - \frac{1}{2}-$, and $N_j > N$ if $u_j$ is of type (\I).

\medskip

\noi $\bullet$ {\bf Cases (k), (k'):} $u_1, u_2, u_3$ of type $(\I)$. \ \ 
In this case, we have 
\begin{align}
	\label{casekb} \eqref{u1b} \lesssim \dl^{\theta} (N^1)^{0+} \bigg( \sum_{|n| \lesssim N^1} \Big|\sum_{B_n} \frac{g_{n_1}}{|n_1|^\al} \frac{\cj{g_{n_2}}}{|n_2|^\al} \frac{g_{n_3}}{|n_3|^\al} \Big|^2 \bigg)^\frac{1}{2}. 
\end{align}

\noi
Note that we have $N_1, N_2, N_3 > N$.
First, we consider the contribution from $n_1 \ne n_3$. 
As in Subsection \ref{SUBSEC:LWP4}, 
by \eqref{chaosestimate} and Lemma \ref{LEM:count1}, we have 
\begin{align*}
	\text{RHS of } \eqref{casekb} & \lesssim \dl^{\theta-\frac{3}{2}\beta} (N^1)^{\frac{3}{2}\eps+} \bigg( \sum_{|n| \lesssim N^1} \sum_{C_n} \frac{1}{|n_1|^{2\al}|n_2|^{2\al}|n_3|^{2\al}} \bigg)^\frac{1}{2}\\
	& \lesssim \dl^{\theta-\frac{3}{2}\beta} (N^1)^{-\al +\frac{3}{2}\eps+} (N^2)^{-\al} (N^3)^{-\al+\frac{1}{2}} \lesssim N^{ - 3\al + \frac{1}{2} +} \leq N^{3s -\g-}\prod_{j = 1}^3 N_j^{0-} 
\end{align*}

\noi for $\al > \frac{1}{3} + \frac{1}{6}\g> \frac{1}{3}$ and sufficiently large $N$ outside an exceptional set of measure
$< (N^1)^{0-} e^{-\frac{1}{\dl^{c}}}$ as in \eqref{EXCEPT1}.

The contribution from $n_1 = n_3$  follows 
as in \eqref{ZK}.
By Lemmata \ref{LEM:count1} and \ref{LEM:prob1}, we have 
\begin{align*}
	\text{RHS of } \eqref{casekb} & \lesssim \dl^{\theta-\frac{3}{2}\beta} (N^1)^{0+}N_1^{-2\al + 2\eps} N_2^{-\al + \eps}(N^3)^\frac{1}{2} \leq N^{3s-\g-}\prod_{j = 1}^3 N_j^{0-} 
\end{align*}

\noi for $\al > \frac{1}{3} + \frac{1}{6}\g> \frac{1}{3}$ and sufficiently large $N$ outside an exceptional set of measure $< e^{-\frac{1}{\dl^c}}$. \medskip

\noi $\bullet$ {\bf Case (a) :} (Cases (b), (a'), and (b') can be treated in a similar way by replacing $n_2$ with $n_3$, $n_2$ with $n_1$, and $(n_1, n_2, n_3)$ with $(n_2, n_3, n_1),$ respectively.)

In this case, we have $\mu = 2 (n_2 - n_1) (n_2 - n_3) = o( (N_2)^{8-16\al+})$.
Thus, by Lemma \ref{LEM:prob1}, Cauchy-Schwarz inequality, and \eqref{A1} as before, we have 
\begin{align*}
	\eqref{u2b} & 
	\lesssim \dl^{\theta-\frac{\beta}{2}} (N_2)^{-\al + \frac{1}{2}\eps +} N^{-2s} 
	\bigg( \sum_n  \Big(\sum_{A_n}|a_1(n_1)|^2|a_3(n_3)|^2\Big)
	 \Big(\sum_{A_n}1\Big)\bigg)^\frac{1}{2} \notag \\
		& \lesssim \dl^{\theta-\frac{\beta}{2}} N_2^{-\al+\eps+} N^{-2s}
	\Big( \sum_n  \sum_{A_n}|a_1(n_1)|^2|a_3(n_3)|^2
	 \Big)^\frac{1}{2} \notag \\
		& \lesssim \dl^{\theta-\frac{\beta}{2}} N_2^{-\al+\eps+} N^{-2s}
		\leq N^{3s-\g-} N_2^{0-} N_3^{0-} 
\end{align*}

\noi 
for $\al > \frac{5}{12}+\frac{1}{6}\g > \frac{5}{12}$
and sufficiently large $N$
outside an exceptional set of measure $< e^{-\frac{1}{\dl^c}}$.

\medskip

\noi $\bullet$ {\bf Case (c):} (Case (d) can be treated in a similar way by replacing $n_2$ and $n_3$.)

 By Lemma \ref{LEM:prob1} and H\"older inequality on $n_3$ in the inner sum, 
\begin{align*}
	\eqref{u1b} \lesssim \dl^{\theta-\frac{\beta}{2}} N_1^{ -\al +\eps+} N^{-2s} 
	 \Big( \sum_{|n| \lesssim N^1} \sum_{B_n} |a_2(n_2)|^2 \Big)^\frac{1}{2} 
\end{align*}

\noi outside an exceptional set of measure $< e^{-\frac{1}{\dl^c}}$. 
For fixed $n_2$, it follows from (the proof of) Lemma \ref{LEM:count1} that there are at most $N_1^{0+}$ terms in the sum. 
Hence, we have 
\begin{align*}
	\eqref{u1b} \lesssim \dl^{\theta-\frac{\beta}{2}} N_1^{ -\al +\eps+} 
	N^{-2s}  \leq N^{3s-\g -} \prod_{j = 1}^3 N_j^{0-}
\end{align*}

\noi 
for $\al > \frac{5}{12}+\frac{1}{6}\g > \frac{5}{12}$
 and sufficiently large $N$.

\medskip

\noi $\bullet$ {\bf Case (e) :} (Case (f) is basically the same.)

In this case, we have $ |\mu| = |2 (n_2 - n_1) (n_2 - n_3)| \ll  N_2^{4-8\al+} $. 
This implies that $|n|, |n_1|, |n_3| \lesssim N_2^q$
for some $q> 0$ since $n_2 \ne n_1, n_3$.
Then, by Lemma \ref{LEM:prob1}, Cauchy-Schwarz inequality, and \eqref{A1} as before, 
we have
\begin{align*}
	\eqref{u2b} 
& \lesssim \dl^{\theta-\beta} 
	N_2^{-\al+\eps+} N_3^{-\al} N^{-s}\bigg( \sum_{|n| \lesssim N_2^q} 
	\Big(\sum_{A_n} |a_1(n_1)|^2\Big) \Big(\sum_{A_n} 1\Big) \bigg)^\frac{1}{2} \\
& \lesssim \dl^{\theta-\beta} 
	N_2^{-\al+\frac{3}{2}\eps+} N_3^{-\al} N^{-s}\Big( \sum_{|n| \lesssim N_2^q} \sum_{A_n} |a_1(n_1)|^2 \Big)^\frac{1}{2} \\
& \lesssim \dl^{\theta-\beta} N_2^{-\al+2\eps+} N_3^{ -\al} N^{-s}
\leq N^{3s-\g-}N_2^{0-} N_3^{0-} 
\end{align*}

\noi for $ \al > \frac{1}{3} + \frac{1}{6}\g>\frac{1}{3}$ and sufficiently large $N$ outside an exceptional set of measure
$<e^{-\frac{1}{\dl^c}}$.

\medskip

\noi $\bullet$ {\bf Case (g) :} (Cases (h), (i), (j) are basically the same.)

By Lemma \ref{LEM:prob1} and Cauchy-Schwarz inequality as before,  we have
\begin{align*}
	\eqref{u1b} 
& \lesssim \dl^{\theta-\beta} 
	N_1^{-\al+\eps+} N_2^{-\al+ \frac{1}{2}\eps} N^{-s}\Big( \sum_{|n| \lesssim N_1} 
	\sum_{A_n} |a_3(n_3)|^2\Big)^\frac{1}{2} \\
& \lesssim \dl^{\theta-\beta} 
	N_1^{-\al+\frac{3}{2}\eps+} N_2^{-\al+ \frac{1}{2}\eps} N^{-s}
 \leq N^{3s-\g-}\prod_{j = 1}^3 N_j^{0-} 
\end{align*}

\noi for $ \al > \frac{1}{3} + \frac{1}{6}\g>\frac{1}{3}$ and sufficiently large $N$ outside an exceptional set of measure
$<e^{-\frac{1}{\dl^c}}$.

\begin{remark} \rm
It is worthwhile to note that 
the worst case occurs:
\begin{itemize}
\item for all type (I) in the local theory.
\item for one (I) and two (II.b) in the global theory.
\end{itemize}

\noi
These cases yield the conditions on the values of $\al$ in Theorems \ref{THM:LWP} and \ref{THM:GWP1}.

\end{remark}

\medskip

\noindent {\bf Acknowledgments:} The authors would like to thank Prof.~Kenji Nakanishi for the conversation at Institut Henri Poincar\'e. 
They are also grateful to Prof.~Nicolas Burq for remarks on a preliminary draft.
Lastly, they would like to express their gratitude to the anonymous referees for thoughtful
comments that have improved this paper. 


	\bibliographystyle{plain}

\end{document}